\documentclass{amsart}
\usepackage{amssymb}
\usepackage{eucal}
\usepackage{amsfonts}
\usepackage{epsfig}
\usepackage{color}
\usepackage[frame,ps,dvips,matrix,arrow,curve,rotate]{xy}
\let\oldcite\cite                                  

\newcommand{\separate}{{\center \rule{9cm}{0cm} \\ \rule{9cm}{0.5pt} \\
\rule{9cm}{0cm}}}

\newtheorem{thm}{Theorem}[section]
\newtheorem{cor}[thm]{Corollary}

\newtheorem{prop}[thm]{Proposition}
\theoremstyle{definition}

\theoremstyle{remark}
\newtheorem{rem}[thm]{Remark}
\numberwithin{equation}{section} \theoremstyle{remark}
\newtheorem{ex}[thm]{Example}

\newcommand{\U}{\mathcal{U}}
\newcommand{\D}{\mathcal{D}}
\newcommand{\OO}{\mathcal{O}}
\newcommand{\F}{\mathcal{F}}
\newcommand{\G}{\mathcal{G}}
\newcommand{\T}{\mathcal{T}}

\newcommand{\C}{\mathsf{C}}
\newcommand{\B}{\mathsf{B}}
\newcommand{\A}{\mathsf{A}}

\newcommand{\bbZ}{\mathbb{Z}}
\newcommand{\bbF}{\mathbb{F}}
\newcommand{\bbT}{\mathbb{T}}
\newcommand{\bbD}{\mathbb{D}}

\let\:=\colon

\newcommand{\al}{\alpha}
\newcommand{\be}{\beta}

\newcommand{\lra}{\longrightarrow}
\newcommand{\llra}[1]{\stackrel{#1}{\lra}}

\newcommand{\im}{\operatorname{im}}
\newcommand{\coker}{\operatorname{coker}}
\newcommand{\End}{\operatorname{End}}
\newcommand{\Hom}{\operatorname{Hom}}

\newcommand{\PreSh}{\mathbf{PreSh}}
\newcommand{\Sh}{\mathbf{Sh}}
\newcommand{\Spec}{\operatorname{Spec}}
\newcommand{\Ch}{\operatorname{Ch}}
\newcommand{\Tor}{\operatorname{Tor}}
\newcommand{\Ext}{\operatorname{Ext}}
\newcommand{\Cyl}{\operatorname{Cyl}}
\newcommand{\Cone}{\operatorname{Cone}}

\newcommand{\id}{\operatorname{id}}

\newcommand{\tul}[1]{\T^{\leq #1}}
\newcommand{\tug}[1]{\T^{\geq #1}}

\newcommand{\taudl}[1]{\tau_{\leq #1}}
\newcommand{\taudg}[1]{\tau_{\geq #1}}

\def\smashedlongrightarrow{\setbox0=\hbox{$\longrightarrow$}\ht0=1pt\box0}
\def\risom{\buildrel\sim\over{\smashedlongrightarrow}}

\newcommand{\torel}[1]{\stackrel{#1}{\to}}

\newcommand{\thar}[7]{\ar    #4  @{#1} @<0.5pt> [#5]   #6       
                      \ar    #4  @{#2}          [#5]   #6   #7     
                      \ar    #4  @{#3} @<-0.5pt> [#5]   #6        
                                           }

\newcommand{\dentar}[3]{\thar{--}{>}{--}{#1}{#2}{}{#3}}
\newcommand{\thickar}[3]{\thar{-}{->}{-}{#1}{#2}{}{#3}}

\begin{document}

\title[Lectures on derived and triangulated categories]
{Lectures on derived and triangulated categories}%
\author{Behrang Noohi}
\email{behrang@alum.mit.edu}       
\address{Mathematics Department\\ 
         Florida State University\\
         208 Love Building\\
         Tallahassee,  FL 32306-4510 \\
         U.S.A.}

\maketitle

These are the notes of three lectures  given in the International
Workshop on Noncommutative Geometry held in I.P.M., Tehran, Iran,
September 11-22.

The first lecture is an introduction to the basic notions of {\em
abelian category theory}, with a view toward their algebraic
geometric incarnations (as categories of modules over rings or
sheaves of modules over schemes).

In the second lecture, we motivate the importance of {\em chain
complexes} and work out some of their basic properties. The emphasis
here is on the notion of  {\em cone} of a chain map, which will
consequently lead to the notion of  an exact triangle of chain
complexes, a generalization of the cohomology long exact sequence.
We then discuss the homotopy category and the {\em derived category}
of an abelian category, and highlight their main properties.

As a way of formalizing the properties of the cone construction, we
arrive at the notion of  a {\em triangulated category}. This is the
topic of the third lecture. After presenting the main examples of
triangulated categories (i.e., various homotopy/derived categories
associated to an abelian category), we discuss the problem of
constructing abelian categories from a given triangulated category
using $t$-structures.

\vspace{0.1in} \noindent {\em A word on style.} In writing these
notes, we have tried to follow a {\em lecture style} rather than an
{\em article style}. This means that, we have tried to be very
concise, keeping the explanations to a minimum, but not less
(hopefully). The reader may find here and there certain remarks
written in small fonts; these are meant to be side notes that can be
skipped without affecting the flow of the material. With a few
insignificant exceptions, the topics are arranged in linear order.

\vspace{0.1in}
\noindent{\em A word on references.} The references given in these
lecture notes are mostly suggestions for further reading and are not
meant for giving credit or to suggest originality.

\vspace{0.1in}
\noindent{\bf Acknowledgement.} I would like to thank Masoud
Khalkhali, Matilde Marcolli, and Mehrdad Shahshahani for organizing
the excellent workshop, and for inviting me to participate and
lecture in it. I am indebted to Masoud for his continual help and
support. At various stages of preparing these lectures, I have
benefited from discussions with Snigdhayan Mahanta .  I  thank
Elisenda Feliu for reading the notes and making useful comments and,
especially for creating the beautiful diagrams that I have used in
Lecture 2!

\newpage

\tableofcontents%

\newpage
{\LARGE\part*{Lecture 1: Abelian categories}}

\vspace{0.3in}

\noindent{\bf Overview.} In this lecture, we introduce {\em
additive} and {\em abelian} categories, and discuss their most basic
properties. We then concentrate on the examples of abelian
categories that we are interested in. The most fundamental example
is the category of {\em modules} over a ring. The next main class of
examples consists of various categories of {\em sheaves of modules}
over a space; a special type of these examples, and a very important
one, is the category of {\em quasi-coherent sheaves} on a scheme.
The idea here is that one can often recover a space from an
appropriate category of sheaves on it. For example, we can recover a
scheme from the category of quasi-coherent sheaves on it. This point
of view allows us to think of an abelian category as ``a certain
category of sheaves on a certain hypothetical space''. One might
also attempt to extract the ``ring of functions'' of this
hypothetical space. Under some conditions this is possible, but the
ring obtained is not unique. This leads to the notion of {\em Morita
equivalence} of rings. The slogan is that the ``ring of functions''
on a noncommutative space is  well defined only up to Morita
equivalence.

\vspace{0.3in}

\section{Products and coproducts in categories}{\label{S:Products}}

\noindent {\small References: \cite{HiSt}, \cite{Weibel},
 \cite{Freyd}, \cite{MacLane},  \cite{GeMa}.}

\vspace{0.1in}

\noindent Let $\C$ be a category and $\{X_i\}_{i\in I}$ a set of
objects in $\C$.

\vspace{0.2in}

\noindent The {\bf product} $\prod_{i \in I} X_i$ is an object in
$\C$, together with a collection of morphisms $\pi_i \: \prod_{i \in
I} X_i \to X_i$, satisfying the following universal property:

\vspace{0.1in}

 \fbox{
Given any collection of morphisms  $f_i \: Y \to X_i$, \ \ \
     $\xymatrix@=16pt@M=8pt@C=28pt{ &  \prod_{i \in I} X_i
      \ar[d]^{\pi_i}   \\
                  Y \ar[r]_{f_i} \ar@{..>}^(0.4){\exists !\, f}[ru] &
                   X_i }$
     }

\vspace{0.3in}

\noindent The {\bf coproduct} $\coprod_{i \in I} X_i$ is an object
in $\C$, together with a collection of morphisms $\iota_i \: X_i \to
\coprod_{i \in I} X_i$, satisfying the following universal property:

\vspace{0.1in}

 \fbox{
Given any collection of morphisms  $g_i \: X_i \to Y$, \ \ \
     $\xymatrix@=16pt@M=8pt@C=28pt{ &
            \coprod_{i \in I} X_i \ar@{..>}_(0.55){\exists !\, f}[dl] \\
                  Y   &   X_i \ar[l]^{g_i} \ar[u]_{\iota_i} }$
      }

\vspace{0.1in}

\begin{rem}
    Products and coproducts may or may not exist, but if they do they
    are unique up to  canonical isomorphism.
\end{rem}

\begin{ex}{\label{E:1}}
\end{ex}
\begin{itemize}
    \item[$\bf{1.}$] \ $\C=\mathbf{Sets}$:  \ \ \ \  $\coprod$= disjoint
    union, \ $\prod$=cartesian product.

    \vspace{0.1in}

    \item[$\bf{2.}$] \ $\C=\mathbf{Groups}$:  \ \ \ \  $\coprod$= free
    product, \
      $\prod$=cartesian product.

    \vspace{0.1in}

    \item[$\bf{3.}$]  \ $\C=\mathbf{UnitalCommRings}$:  \ \ \ \
    $\underset{\text{finite}}{\coprod}$=
    $\otimes_{\mathbb{Z}}$, \ $\prod$=cartesian product.


    \item[$\bf{4.}$]  \ $\C=\mathbf{Fields}$:  \ \ \ \  $\coprod$= does not exist,
      \ $\prod$=does not exist.

    \vspace{0.1in}

    \item[$\bf{5.}$] \ $\C=R$-$\mathbf{Mod}$=left $R$-modules, $R$ a ring:  \ \ \
      $\underset{\text{finite}}{\coprod}=\underset{\text{finite}}{\prod}=\oplus$,
      the \\  usual direct  sum  of modules.
\end{itemize}

\vspace{0.1in}

 \noindent  {\em Exercise.}  Show that in $R$-$\mathbf{Mod}$ there is a natural morphism
  $\coprod_{i \in I} X_i \to  \prod_{i \in I} X_i$, and give an
  example where this is not an isomorphism.

\section{Abelian categories}{\label{S:Additive}}

\noindent {\small References: \cite{HiSt}, \cite{Weibel},
\cite{Freyd}, \cite{MacLane},   \cite{GeMa}.}

\vspace{0.1in}

 We discuss three types of categories: {\em $Ab$-categories}, {\em
additive categories}, and {\em abelian categories}. Each type of
category has more structure/properties than the previous one.

\separate

\noindent An {\bf $Ab$-category} is a category $\C$ with the
following {\em extra structure}: each $\Hom_{\C}(X,Y)$ is endowed
with the structure of an abelian group. We require that composition
is linear:
  $$\xymatrix@=16pt@M=8pt@C=28pt{ X \ar[r]^u &  Y \ar@<1ex> [r]^f
     \ar@<-1ex> [r]_g & Z \ar[r]^v & T }$$
    $$u(f+g)=uf+ug, \ \ \ (f+g)v=fv+gv.$$

\vspace{0.1in}

\noindent An {\bf additive functor} between $Ab$-categories is  a
functor that induces group homomorphisms on $\Hom$-sets.

\vspace{0.2in}

\noindent An {\bf additive} category is a special type of
$Ab$-category. More precisely, an  additive category is an
$Ab$-category with the following {\em properties}:

\vspace{0.1in}
  \begin{itemize}
     \item[$\blacktriangleright$] There exists a {\em zero} object $0$
       such that
          $$\forall X, \ \ \Hom_{\C}(0,X)=\{0\}=\Hom_{\C}(X,0).$$

     \vspace{0.1in}

     \item[$\blacktriangleright$] Finite sums exist. (Equivalently,
     finite products exist; see Proposition \ref{P:sum}.)
  \end{itemize}

\separate

Before defining abelian categories, we discuss some basic facts
about $Ab$-categories.

\begin{prop}{\label{P:sum}} Let $\C$ be an $Ab$-category with a zero
 object, and let $\{X_i\}_{i\in I}$ be a finite set of objects in $\C$.
 Then $\coprod X_i$ exists if an only if $\prod X_i$ exists. In this case,
 $\coprod X_i$ and  $\prod X_i$ are naturally isomorphic.
\end{prop}

\begin{proof}
   Exercise. (Or See \cite{HiSt}.)
\end{proof}

\vspace{0.1in}

\noindent {\em Notation.} Thanks to Proposition \ref{P:sum}, it
makes sense to use the same symbol $\oplus$ for both product and
coproduct (of finitely many objects) in an additive category.

\vspace{0.2in}

\noindent {\em Exercise.}  In Example \ref{E:1}, which ones can be
made into additive categories?

\separate

\noindent A morphism  $f \: B \to C$ in an additive category is
called a {\bf monomorphism} if

\vspace{0.1in}

     \begin{center}  \fbox{$\forall \ X \llra{g} B \llra{f} C, \  \ \
                                         f\circ g=0 \ \ \ \Rightarrow
     \ \     g=0$.}\end{center}

\vspace{0.1in}

\noindent A {\bf kernel} of a morphism $f \: B \to C$ in an additive
category is a morphism $i \: A \to B$ such that  $f\circ i=0$, and
that for every $g \: X \to B$ with $f\circ g=0$

\vspace{0.1in}

 \begin{center}\fbox{
       $\xymatrix@=16pt@M=8pt@C=28pt{ A \ar[r]^i & B \ar[r]^f & C \\
                        & X \ar@{..>} [lu]^{\exists !} \ar[ru]_0 \ar[u]^(0.45){g} &
                        }$}
 \end{center}

\vspace{0.1in}

\begin{prop}
If $i$ is a kernel for some morphism, then $i$ is a
 monomorphism. The converse is not always true.
\end{prop}

  We can also make the dual definitions.

\vspace{0.1in}

\noindent A morphism  $f \: B \to C$ in an additive category is
called an {\bf epimorphism} if

\vspace{0.1in}

     \begin{center}  \fbox{$\forall \ B \llra{f} C \llra{h} D,  \  \ \ h\circ f=0 \ \ \ \Rightarrow
          \ \     h=0$.}
     \end{center}

\vspace{0.1in}

\noindent A {\bf cokernel} of a morphism $f \: B \to C$ in an
additive category is a morphism $p \: C \to D$ such that  $p\circ
f=0$, and such that for every $h \: B \to Y$ with $h\circ f=0$

\vspace{0.1in}

 \begin{center}\fbox{
       $\xymatrix@=16pt@M=8pt@C=28pt{ B  \ar[rd]_0 \ar[r]^f &
            C \ar[d]_(0.4){h} \ar[r]^p & D \ar@{..>} [ld]^{\exists !} \\
                        & Y   &            }$}
 \end{center}

\vspace{0.1in}

\begin{prop} If $p$ is a cokernel for some morphism, then $p$ is an
 epimorphism. The converse is not always true.
\end{prop}

\begin{rem}
    Kernels and cokernels may or may not exist, but if they do they
    are unique up to a canonical isomorphism.
\end{rem}

\begin{ex}{\label{E:2}}
\end{ex}
\begin{itemize}
    \item[$\bf{1.}$] In $\A=R$-$\mathbf{Mod}$ kernels
       and cokernels always exist.

\vspace{0.1in}

    \item[$\bf{2.}$] Let $\A$ be the category of finitely
    generated free $\mathbb{Z}$-modules. Then kernels and cokernels
    always exist.
     ({\em Exercise.} What is the cokernel of $\mathbb{Z}\llra{\times 3} \mathbb{Z}$?)

\vspace{0.1in}

    \item[$\bf{3.}$] Let $\A$ be the category of
    $\mathbb{C}$-vector spaces of even dimension. Then kernels and
    cokernels do not always exist in $\A$. (Give an example.)
    The same thing is true if $\A$ is the category of infinite dimensional
    vector spaces.

\end{itemize}

\separate

\noindent An {\bf abelian} category is an additive category $\A$
  with the following properties:

\vspace{0.1in}

\begin{itemize}
    \item[$\blacktriangleright$] Kernels and cokernels always exist
      in $\A$.

\vspace{0.1in}

    \item[$\blacktriangleright$] Every monomorphism is a kernel and
      every epimorphism is a cokernel.

\end{itemize}

\vspace{0.1in}

\noindent{\bf Main example.} For every ring $R$, the additive
category $R$-$\mathbf{Mod}$ is abelian.

\begin{rem}
Note that an $Ab$-category is a category {\em with an extra
structure}. However, an additive category is just an $Ab$-category
which satisfies some property (but no additional structure). An
abelian category is an additive category which satisfies some more
properties.
\end{rem}


\noindent{\em Exercise.} In Example \ref{E:2} show that
($\mathbf{1}$) is abelian, but ($\mathbf{2}$) and ($\mathbf{3}$) are
not. In ($\mathbf{2}$) the map $\mathbb{Z}\llra{\times 3}
\mathbb{Z}$ is an epimorphism that is not the cokernel of any
morphism, because it is also a monomorphism!

\vspace{0.1in}

\begin{prop}
   Let $f \: B \to C$ be a morphism in an abelian category. Let
   $i \: \ker(f) \to B$ be its kernel and $p \: C \to \coker(f)$ its
   cokernel. Then there is a natural isomorphism $\coker(i)\risom
   \ker(p)$ fitting in the following commutative diagram:
        $$\xymatrix@=10pt@M=8pt@C=14pt{ \ker(f) \ar@{^(->} [r]^(0.55)i & B
        \ar@{->>} [d]         \ar[r]^f & C \ar@{->>} [r]^(0.35)p & \coker(f) \\
                  & \coker(i)  \ar[r]^{\sim} &   \ker(p) \ar@{^(->} [u]
                  &}$$
\end{prop}

\begin{cor}
  Every morphism $f \: B \to C$ in an abelian category has a unique
  (up to a unique isomorphism) factorization
      $$\xymatrix@=14pt@M=6pt@C=10pt{B \ar[rr]^f \ar@{->>} [dr]_{f_{epi}} & & C \\
        & I \ar@{^(->} [ru]_{f_{mono}} &    }$$
\end{cor}

\begin{cor}
   In an abelian category, {\em mono + epi $\Leftrightarrow$ iso}.
\end{cor}

\vspace{0.1in}

\noindent The object $I$ (together with the two morphisms $f_{epi}$
and $f_{mono}$) in the above corollary is called the {\bf image} of
$f$ and is denoted by $\operatorname{im}(f)$. The morphism
$f_{mono}$ factors through every monomorphism into $C$ through which
$f$ factors. Dually,  $f_{epi}$ factors through every epimorphism
originating from $B$ through which $f$ factors. Either of these
properties characterizes the image.

\vspace{0.1in}%

\noindent In an abelian category, a sequence
  $$ 0 \llra{} A \llra{f} B \llra{g} C \llra{} 0$$
is called {\bf exact} if $f$ is a monomorphism, $g$ is an
epimorphism, and $\operatorname{im}(f)=\ker(g)$. An additive functor
between abelian categories is called exact if it takes exact
sequences to exact sequences.

\separate

We saw that $R$-modules form an abelian category for every ring $R$.
In fact, every small abelian category is contained in some
$R$-$\mathbf{Mod}$. (A category is called {\em small} if its objects
form a set.)

\begin{thm}[Freyd-Mitchell Embedding Theorem \oldcite{Freyd,Mitchel1}]
 Let $\A$ be a small abelian category. Then there exists a unital ring
 $R$ and an exact fully faithful functor $\A \to R$-$\mathbf{Mod}$.
\end{thm}

\section{Categories  of sheaves}{\label{S:Sheaves}}

\noindent {\small References: \cite{Hartshorne1}, \cite{KaSch},
\cite{Iversen}, \cite{GeMa}.}

\vspace{0.1in}

The main classes of examples of abelian categories are categories of
sheaves  over spaces. We give a quick review of sheaves and describe
the abelian category structure on them.

\separate

\noindent Let $\C$ be an arbitrary category (base) and $\A$ another
category (values).

\vspace{0.1in}

\noindent A {\bf presheaf} on $\C$ with values in $\A$ is a functor
$\F \: \C^{op} \to \A$. A morphism $f \: \F \to \G$ of presheaves is
a natural transformation of functors. The category of presheaves is
denoted by $\PreSh(\C,\A)$.

\vspace{0.1in}

\noindent {\bf Typical example.} Let $X$ be a topological space, and
let $\C=\mathsf{Open}_X$ be the category whose objects are open sets
of $X$ and whose morphisms are inclusions. Let $\A=\mathsf{Ab}$, the
category of abelian groups. A presheaf $\F$ of abelian groups on $X$
consists of:

\vspace{0.1in}

\begin{itemize}
  \item[$\triangleright$] A collection of abelian groups
       $$ \F(U), \ \ \forall \ U \ \ \text{open};$$


  \item[$\triangleright$] ``Restriction'' homomorphisms
     $$\F(U) \to \F(V), \ \ \forall \ V \subseteq U.$$
     Restriction homomorphisms should respect triple inclusions $W\subseteq
     V\subseteq U$ and  be equal to the identity for $U\subseteq U$.
\end{itemize}

\vspace{0.1in}

\begin{ex}{\label{E:3}}
\end{ex}
\begin{itemize}
    \item[$\bf{1.}$] {\em (Pre)sheaf of continuous functions on a
                                       topological space $X$.}
       The assignment
     $$U \mapsto  \OO_X^{cont}(U)=\{\text{continuous
                                functions on} \ \ U\}$$
                     is a presheaf on $X$.
                     The restriction maps are simply restriction of
                     functions. In fact, this is a presheaf of {\em rings}
                     because restriction maps are ring
                     homomorphisms.

        \vspace{0.1in}

    \item[$\bf{2.}$] {\em Constant presheaf.}  Let $A$ be an abelian
         group.  The assignment
                 $$U \mapsto A$$
        is a presheaf of groups. The restriction maps are the
        identity maps.
\end{itemize}

\vspace{0.1in}

{\small \noindent{\em Variations.} There are many variations on
these examples. For instance, in ($\mathbf{1}$) one can take the
(pre)sheaf of $C^{r}$ functions on a $C^r$-manifold, or (pre)sheaf
of holomorphic functions on a complex manifold, and so on. These are
called {\em structure sheaves}. {\em Idea:} Structure sheaves encode
all the information about the structure in question (e.g, $C^r$,
analytic, holomorphic, etc.). So, for instance, a complex manifold
$X$ can be recovered from its {\em underlying topological space}
$X^{top}$ and the {\em sheaf} $\OO_X^{holo}$. That is, we can think
of the pair $(X^{top},\OO_X^{holo})$ as a complex manifold. }

\vspace{0.1in}

\noindent {\em Exercise.} Formulate the notion of a holomorphic map
of complex manifolds purely in terms of the pair
$(X^{top},\OO_X^{holo})$.

\vspace{0.1in}

\begin{prop}
  Let $\C$ be an arbitrary  category, and $\A$ an abelian category.
  Then
  $\PreSh(\C,\A)$ is an abelian category.
\end{prop}

\noindent The kernel and cokernel of a morphism  $f \: \F \to \G$ of
presheaves are given by
        \vspace{0.1in}

\begin{center}
\fbox{
\begin{minipage}{2.4in}
 \ \ \ $\ker(f) \: \ \ U \mapsto \ker\big(\F(U)\llra{f_U}\G(U)\big)$ \\
  $\coker(f) \: \ \ U \mapsto \coker\big(\F(U)\llra{f_U}\G(U)\big)$
\end{minipage}}
\end{center}

\separate

\noindent A presheaf $\F$, say of abelian groups, rings etc., on $X$
is called a {\bf sheaf} if for every open $U \subseteq X$ and every
open cover $\{U_{\al}\}$ of $U$ the sequence

        \vspace{0.1in}

\begin{center}
\fbox{
\begin{minipage}{2.6in}
  $\F(U) \llra{res} \prod_{\al} \F(U_{\al}) \lra \prod_{\al,\be}\F(U_{\al}\cap
  U_{\be})^{\hbox{}}
  \vspace{0.1in} \\
  \hbox{} \ \ \ \ \ \ \ \ \ \ \ \  \ \ \  (f_{\al}) \mapsto (f_{\al}|_{U_{\al}\cap U_{\be}} -
             f_{\be}|_{U_{\al}\cap U_{\be}})$
\end{minipage}}
\end{center}
is exact.

\vspace{0.1in}

\begin{ex} Structure sheaves (e.g., Example \ref{E:3}.$\mathbf{1}$) are
  sheaves! More generally, every vector bundle $E \to X$ gives rise
  to a sheaf of vector spaces via the assignment $U \mapsto E(U)$, where $E(U)$
  stands for the space of sections of $E$ over $U$. Is a constant
  presheaf (Example \ref{E:3}.$\mathbf{2}$) a sheaf?
\end{ex}

\vspace{0.1in}

\noindent For an abelian category $\A$, we denote the
 full subcategory of $\PreSh(X,\A)$
whose objects are sheaves by  $\Sh(X,\A)$.

\vspace{0.1in}

\begin{prop}
  The category $\Sh(X,\A)$ is abelian.
\end{prop}

\begin{rem}
\end{rem}
\begin{itemize}
  \item [1.] Monomorphisms in $\Sh(X,\A)$ are the same as the ones
      in $\PreSh(X,\A)$, but epimorphisms are different: $f \: \F
      \to \G$ is an epimorphism if for every open $U$ and every $a \in
      \G(U)$, \\
      \begin{center}
      \fbox{
        \begin{minipage}{3.2in}
          $\exists \ \{U_{\al}\},  \ \text{open cover of}  \ U,
          \ \text{such that:} \\ \forall\al, \ a|_{U_{\al}}  \ \text{is
         in the image of} \ \ f({U_\al}) \: \F(U_{\al}) \to \G(U_{\al}).$
        \end{minipage}}
      \end{center}

\vspace{0.1in}

  \item[2.] Kernels in $\Sh(X,\A)$ are defined in the same way as
     kernels in $\PreSh(X,\A)$, but cokernels are defined
     differently: if $f \: \F \to \G$ is a morphism of sheaves,
      the cokernel of $f$ is the

\vspace{0.1in}

      \begin{center}
       \fbox{
        \begin{minipage}{3.8in}
            sheaf associated to the presheaf \ \ $U \mapsto
            \coker\big(\F(U)\llra{f_U}\G(U)\big)$.
        \end{minipage}}
      \end{center}
\end{itemize}

\vspace{0.1in}

  We remark that there is a general procedure
for producing a sheaf $\F^{sh}$ out of a presheaf $\F$. This is
called {\em sheafification}. The sheaf $\F^{sh}$ is called the {\em
sheaf associated} to $\F$ and it comes with a natural morphism of
presheaves $i\: \F \to \F^{sh}$ which is universal among morphisms
$\F \to \G$ to sheaves $\G$. More details on this can be found in
\cite{Hartshorne1}.

\section{Abelian category of quasi-coherent sheaves on a scheme}{\label{S:Schemes}}

\noindent {\small References: \cite{Hartshorne1},  \cite{GeMa}.}

\vspace{0.1in}

We give a super quick review of schemes. We then look at the
category of quasi-coherent sheaves on a scheme.

\separate

\noindent The {\bf affine scheme} $(\Spec R, \OO_R)$  associated to
a commutative unital ring $R$ is a topological space $\Spec R$, the
space of prime ideals in $R$, together with a sheaf of rings
$\OO_R$, the {\em structure sheaf}, on it. Recall that, in the
topology of $\Spec R$ (called the {\em Zariski topology}) a closed
set is the set $V(I)$ of all prime ideals containing a given ideal
$I$. The sheaf $\OO_R$ is uniquely determined by the fact that, for
every $f \in R$, the ring of sections of $\OO_R$ over the open set
$U_f:=\Spec R - V(f)$ is the localization $R_{(f)}$, that is,
$\OO_R(U_f)=R_{(f)}$. In particular, $\OO_R(\Spec R)=R$. (Note: the
open sets $U_f$ form a basis for the Zariski topology on $\Spec R$.)

\vspace{0.1in}

\noindent A {\bf  scheme} $X$ consists of a pair $(X^{Zar},\OO_X)$
where $X$ is a topological space and $\OO_X$ is a sheaf of rings on
$X^{Zar}$. We require that $X$ can be covered by open sets $U$ such
that each $(U,\OO_U)$ is isomorphic to an affine scheme. (Here,
$\OO_U$ is the restriction of $\OO_X$ to $U$.)

\vspace{0.1in}

\begin{rem} This definition is modeled on Example
\ref{E:3}.$\mathbf{1}$; in particular, read the paragraph after the
example.
\end{rem}

\separate

\noindent A {\bf sheaf of modules} over a scheme $X$ is a sheaf $\F$
  of abelian groups on $X^{Zar}$ such that, for every open $U \subseteq
  X^{Zar}$, $\F(U)$ is endowed with an $\OO_X(U)$-module structure (and restriction maps respect
  the module structure).

\vspace{0.1in}

\noindent A sheaf of modules $\F$ over $X$ is called {\bf
  quasi-coherent} if for every  inclusion of the form
  $\Spec S=V \subseteq U=\Spec R$ of open sets in $X^{Zar}$
  we have

\vspace{0.1in}

   \begin{center}
       \fbox{
              $\F(V)\cong S\otimes_R \F(U)$.}
   \end{center}

\vspace{0.1in}

\begin{prop}
  The category $\OO_X$-$\mathbf{Mod}$ of $\OO_X$-modules on a scheme
  $X$ is an abelian category. The full subcategory
  $\mathbf{Quasi}$-$\mathbf{Coh}_X$ is also an abelian category.
\end{prop}

\begin{ex}
   To an  $R$-module $M$ there is  associated a quasi-coherent sheaf
   $\tilde{M}$ on $X=\Spec R$ which is characterized by the property
   that $\tilde{M}(U_f)=M_{(f)}$, for every $f \in R$. In fact, every
   quasi-coherent sheaf on $\Spec R$ is of this form. More
   precisely, we have an equivalence of categories
   (\cite{Hartshorne1},
   Corollary II.5.5)

        \vspace{0.1in}

   \begin{center}
       \fbox{
              $\mathbf{Quasi}$-$\mathbf{Coh}_{\Spec R} \cong R$-$\mathbf{Mod}$.}
   \end{center}

\end{ex}

\separate

The category $\mathbf{Quasi}$-$\mathbf{Coh}_X$ is a natural abelian
category associated with a scheme $X$. This allows us to do
homological algebra on schemes (e.g.,
 {\em sheaf cohomology}). The following
reconstruction theorem states that, indeed,
$\mathbf{Quasi}$-$\mathbf{Coh}_X$ captures all the information about
$X$.

\vspace{0.1in}

\begin{prop}[Gabriel-Rosenberg Reconstruction Theorem \oldcite{Ro}]
  A scheme $X$ can be reconstructed, up to isomorphism, from
  the abelian category $\mathbf{Quasi}$-$\mathbf{Coh}_X$.
\end{prop}

More generally, Rosenberg \cite{Ro}, building on the work of
Gabriel, associates to an abelian category $\A$ a topological space
$\Spec\A$ together with a sheaf of rings $\OO_{\A}$ on it. In the
case where $\A=\mathbf{Quasi}$-$\mathbf{Coh}_X$, the pair
$(\Spec\A,\OO_{\A})$ is naturally isomorphic to  $(X^{Zar},\OO_X)$.

The above theorem is a starting point in non-commutative algebraic
geometry. It means that one can think of the abelian category
$\mathbf{Quasi}$-$\mathbf{Coh}_X$ itself as a space.

\section{Morita equivalence of rings}

\noindent {\small References:  \cite{Weibel},  \cite{GeMa}.}

\vspace{0.1in}

We saw that, by the Gabriel-Rosenberg Reconstruction Theorem, we can
regard the abelian category $\mathbf{Quasi}$-$\mathbf{Coh}_X$ as
being ``the same'' as the scheme $X$ itself. The Gabriel-Rosenberg
Reconstruction Theorem has a more classical precursor.

\begin{thm}[Gabriel \oldcite{Gabriel,Freyd}]
   Let $R$ be a unital ring (not necessarily commutative).
   Then $R$-$\mathbf{Mod}$ has a small projective
   generator (e.g., $R$ itself), and is closed under arbitrary coproducts.
   Conversely,  let $\A$ be an
   abelian category with a small projective generator $P$ which is closed
   under arbitrary
   coproducts. Let $R=(\End P)^{op}$. Then $\A\cong R$-$\mathbf{Mod}$.
\end{thm}

{\small
\begin{rem} This, however, is not exactly a reconstruction
theorem: the projective generator $P$ is never unique, so we obtain
various rings $S=\End P$ such that $\A\cong S$-$\mathbf{Mod}$.
Nevertheless, all such rings are regarded as giving  the {\em same}
``noncommutative scheme''. So, from the point of view of
noncommutative geometry they are the same. By the Gabriel-Rosenberg
Reconstruction Theorem, if such $R$ and $S$ are both commutative,
then they are necessarily isomorphic.
\end{rem} }

\noindent Two  rings $R$ and $S$ are called {\bf Morita equivalent}
if $R$-$\mathbf{Mod}\cong S$-$\mathbf{Mod}$.

\begin{ex} For any ring $R$, the ring $S=M_n(R)$ of $n\times n$
  matrices over $R$ is Morita equivalent to $R$. ({\em Proof.} In the
  above theorem
  take $P=R^{\oplus n}$, or use the next theorem.)
\end{ex}

\vspace{0.1in}

The following characterization of Morita equivalence is important.
The proof is not hard.

\begin{thm}
  Let $R$ and $S$ be rings. Then the following are equivalent:
   \begin{itemize}
     \item[$\mathbf{i.}$] The categories $R$-$\mathbf{Mod}$
       and $S$-$\mathbf{Mod}$ are equivalent.

     \item[$\mathbf{ii.}$] There is an $S-R$ bimodule $M$ such that
       the functor $M\otimes_R- \: R$-$\mathbf{Mod}\to
       S$-$\mathbf{Mod}$ is an equivalence of categories.

     \item[$\mathbf{iii.}$] There is a finitely generated projective
       generator $P$ for $R$-$\mathbf{Mod}$ such that $S\cong\End P$.
   \end{itemize}
\end{thm}

\section{Appendix: injective and projective objects in abelian categories}

We will need to deal with injective and projective objects in the
next lecture, so we briefly recall their definition.

\separate

\noindent An object $P$ in an abelian category is called {\bf
projective} if it has the following lifting property:

\begin{center}
\fbox{
\begin{minipage}{1.2in}
    $$\xymatrix@=16pt@M=8pt@C=28pt{ & B \ar@{->>}[d]^{\text{epi}} \\
        P \ar[r] \ar@{..>}[ur]^{\exists} & C
    }$$
\end{minipage}}
\end{center}

        \vspace{0.1in}

 \noindent Equivalently, $P$ is projective if the
functor
 $\Hom_{\A}(P,-) \: \A \to \mathsf{Ab}$ takes exact sequences in
 $\A$ to exact sequences of abelian groups.

\vspace{0.1in}

\noindent An object $I$ in an abelian category is called {\bf
injective} if it has the following extension property:

\begin{center}
\fbox{
\begin{minipage}{1.2in}
    $$\xymatrix@=16pt@M=8pt@C=28pt{ & B  \ar@{..>}[dl]_{\exists}  \\
     I  &  A  \ar@{^(->}[u]_{\text{mono}} \ar[l]
    }$$
\end{minipage}}
\end{center}

        \vspace{0.1in}

 \noindent Equivalently, $I$ is injective if the
functor
 $\Hom_{\A}(-,I) \: \A^{op} \to \mathsf{Ab}$ takes exact sequences in
 $\A$ to exact sequences of abelian groups.

\separate

\noindent We say that $\A$ has {\bf enough projectives}
(respectively, {\bf enough injectives}), if for every object $A$
there exists an epimorphism $P \to A$ where $P$ is projective
(respectively, a monomorphism $A \to I$ where $I$ is injective).

\vspace{0.1in}

\noindent In a category with enough projectives (respectively,
enough injectives) one can always find {\em projective resolutions}
(respectively, {\em injective resolutions})  for objects.

\vspace{0.1in}

\noindent The category  $R$-$\mathbf{Mod}$ has enough injectives and
enough projectives. The categories $\PreSh(\C,\mathsf{Ab})$,
$\PreSh(\C,R$-$\mathbf{Mod}$), $\Sh(\C,\mathsf{Ab})$,
$\Sh(\C,R$-$\mathbf{Mod}$), and $\mathbf{Quasi}$-$\mathbf{Coh}_X$
have enough injectives, but in general they do not have enough
projectives.

\vspace{0.1in}

\noindent{\em Exercise.} Show that the abelian category of finite
abelian groups has no injective or projective object other than 0.
Show that in the category of vector spaces every object is both
injective and projective.

\newpage
{\LARGE\part*{Lecture 2: Chain complexes}}

\vspace{0.3in}

\noindent{\bf Overview.} Various cohomology theories in mathematics
 are constructed from chain complexes. However, taking the cohomology of
 a chain complex kills a lot of information contained in that chain complex.
 So, it is desirable to
 elevate the cohomological constructions to the chain complex level.
 In this lecture, we introduce the necessary machinery for doing so.
 The main tool here is the {\em mapping cone} construction, which should
 be thought of as a {\em homotopy cokernel} construction. We
 discuss in some detail the basic properties of the cone
 construction; we will see, in particular, how it allows us to construct long
 exact sequences on the chain complex level, generalizing the usual
 cohomology long exact sequences.

 In practice, one is only interested in chain complexes up to {\em
 quasi-isomorphism}. This leads to the notion of the {\em derived
 category} of chain complexes. The cone construction is well adapted
 to the derived category. Derived categories provide the correct
 setting for manipulating chain complexes; for instance, they allow
 us to construct {\em derived functors} on the chain complex level.

\setcounter{section}{0}

\vspace{0.3in}

\section{Why chain complexes?}

Chain complexes arise naturally in many areas of mathematics. There
are two main sources for (co)chain complexes:  {\em chains on
spaces} and {\em resolutions}. We give an example for each.

\separate

\noindent {\bf Chains on spaces.}
 Let $X$ be a topological space. There are various ways to
associate a chain complex to $X$. For example, {\em singular
chains}, {\em cellular chains} (if $X$ is a CW complex), {\em
simplicial chains} (if $X$ is triangulated), de Rham complex (if $X$
is differentiable), and so on. These complexes encode a great deal
of topological information about $X$ in terms of algebra (e.g.,
homology and cohomology).

\vspace{0.1in}

Observe the following two facts:

\vspace{0.1in}

\begin{itemize}

  \item Such constructions usually give rise to a
                  chain complex of {\em free} modules.

\vspace{0.1in}

  \item Various chain complexes associated to a given space
       are chain homotopy equivalent (hence, give rise to the same
        homology/cohomology).
\end{itemize}

\vspace{0.1in}

\noindent {\bf Resolutions.} Let us explain this with an example.
Let $R$ be a ring, and $M$ and $N$ $R$-modules. Recall how we
compute $\Tor_i(M,N)$. First, we choose a  free resolution
   $$\underbrace{\cdots \to P^{-2} \to P^{-1} \to P^0}_{P^{\bullet}} \to M.$$
Then, we define
   $$\Tor_i(M,N)=H^{-i}(P^{\bullet}\otimes N).$$

\vspace{0.1in}

Observe the following two facts:

\vspace{0.1in}

\begin{itemize}

  \item The complex  $P^{\bullet}$ is a complex of {\em free} modules.

\vspace{0.1in}

  \item Any two resolutions $P^{\bullet}$ and $Q^{\bullet}$ of $M$
       are chain homotopy equivalent (hence $\Tor$ is well-defined).
\end{itemize}

\separate

\noindent{\em Conclusion.} In both examples above, we replaced our
object with a chain complex of free modules that was unique up to
chain homotopy. We could then extract information about our object
by doing algebraic manipulations (e.g., taking homology) on this
complex. Therefore, {\em the real object of interest is the (chain
homotopy class of) a chain complex (of, say, free modules).} Of
course, instead of working with the complex itself, one could choose
to work with its (co)homology, but one loses some information this
way.

\vspace{0.1in}

\begin{rem}
   Complexes of {\em projective} modules (e.g,
   projective resolutions) work equally well as complexes of free
   modules. Sometimes, we are in a dual situation where complexes of
   {\em injective} modules are more appropriate (e.g., in computing
   $\Ext$ groups).
\end{rem}

\section{Chain complexes}

\noindent {\small References: \cite{Weibel}, \cite{GeMa},
\cite{Hartshorne2}, \cite{HiSt}, \cite{KaSch}, \cite{Iversen}, and
any book on algebraic topology.}

\vspace{0.1in}

We quickly recall a few definitions. We prefer to work with {\em
cohomological} indexing, so we work with {\em cochain} complexes.

\separate

\noindent A {\bf cochain complex} $C$ in an abelian category $\A$ is
 a sequence of objects in $\A$
   $$\cdots \llra{} C^{n-1} \llra{d} C^n \llra{d} C^{n-1} \llra{}
   \cdots,
   \ \ \ \ d^2=0.$$
Such a sequence is in general indexed by $\mathbb{Z}$, but in many
applications one works with complexes that are bounded below,
bounded above, or bounded on both sides.

\vspace{0.1in}

\noindent A {\bf chain map} $f \: B \to C$ between two cochain
 complexes is a sequence of $f^n \: B^n \to C^n$ of morphisms
 such that the following diagram commutes

  $$\xymatrix@=22pt@M=6pt@C=16pt{
     \cdots \ar[r]^(0.4){d^{n-2}} &  B^{n-1}  \ar[r]^{d^{n-1}}  \ar[d]^{f^{n-1}} &
            B^{n} \ar[r]^(0.4){d^{n}} \ar[d]^{f^{n}} &
                         B^{n+1} \ar[d]^{f^{n+1}} \ar[r]^{d^{n+1}}  & \cdots  \\
     \cdots \ar[r]^(0.4){d^{n-2}} &    C^{n-1} \ar[r]^{d^{n-1}} &
                   C^{n} \ar[r]^(0.4){d^{n}} & C^{n+1} \ar[r]^{d^{n+1}}  & \cdots   }$$

\vspace{0.1in}

\noindent A {\bf null homotopy} for a chain map $f \: B \to C$ is a
  sequence $s^n \: B^n \to C^{n-1}$ such that
   \vspace{0.1in}

 \begin{center}\fbox{$f^n=s^{n+1}\circ d^n + d^{n-1}\circ s^n, \ \ \forall n$}
  \end{center}

\vspace{0.1in}

\noindent We say two chain maps $f,g \: B \to C$ are {\bf chain
homotopic} if $f-g$ is null homotopic. We say $f \: B \to C$ is a
{\bf chain homotopy equivalence} if there exists $g \: C \to B$ such
that $f\circ g$ and $g\circ f$ are chain homotopy equivalent to the
corresponding identity maps.

\vspace{0.1in}

\noindent To any chain complex $C$ one can associate the following
objects in $\A$:

\begin{center}
\fbox{
\begin{minipage}{1.9in}
             $Z^n(C)=\ker(C^n\llra{d}C^{n+1})$ \\
                       $B^n(C)=\im(C^{n-1}\llra{d}C^{n})$ \\

\vspace{-0.1in}

                       $H^n(C)=Z^n(C)/B^n(C)$
\end{minipage}}
\end{center}

\vspace{0.1in}

\noindent {\em Exercise.} Show that a chain map $f \: B \to C$
induces morphisms on each of the above objects. In particular, we
have induced morphisms $H^n(f) \: H^n(B) \to H^n(C)$ for all $n$.
Show that if $f$ and $g$ are chain homotopic, then they induce the
same map on cohomology. In particular, a chain homotopy equivalence
induces isomorphisms on cohomologies.

\vspace{0.1in}

\noindent A chain map $f \: B \to C$ is called a {\bf
quasi-isomorphism} if it induces isomorphisms on all cohomologies.
(So every chain homotopy equivalence is a quasi-isomorphism.)

\vspace{0.1in}
 \noindent {\em Exercise.} When is a chain complex $C$
quasi-isomorphic to the zero complex? (Hint: exact.) When is $C$
chain homotopy equivalent to the zero complex? (Hint: split exact.)

\vspace{0.1in}

\noindent A {\bf short exact sequence} of chain complexes is a
sequence
  $$0 \to A \to B \to C \to 0$$
of chain maps which is exact at every $n \in \mathbb{Z}$. Such  an
exact sequence gives rise to a long exact sequence in cohomology
{\small
$$\cdots \overset{\partial}{\to} H^n(A) \to H^n(B) \to H^n(C)
          \overset{\partial}{\to} H^{n+1}(A) \to H^{n+1}(B) \to H^{n+1}(C)
                                            \overset{\partial}{\to} \cdots$$
}

\noindent A short exact sequence as above is called {\bf pointwise
split} (or {\bf semi-split}) if every epimorphism $B^n \to C^n$
admits a section, or equivalently, every monomorphism $A^n \to B^n$
is a direct summand.

\vspace{0.1in}
 \noindent {\em Exercise.} Does pointwise split imply that the maps
 $\partial$ in the above long exact sequence are zero? Is the converse true?
 (Answer: both implications are false.)

\section{Constructions on chain complexes}

\noindent {\small References: \cite{Weibel}, \cite{GeMa},
 \cite{Hartshorne2}, \cite{HiSt},
\cite{KaSch}, \cite{Iversen}.}

\vspace{0.1in}

In the following table we list a few constructions that are of great
importance in algebraic topology. The left column presents the
topological construction, and the right column is the cochain
counterpart.

 The right column is obtained from the left column as follows:
imagine that the spaces are triangulated and translate the
topological constructions in terms of  the simplices; this gives the
{\em chain complex} picture. The {\em cochain complex} picture is
obtained by the appropriate change in indexing (from homological to
cohomological).

 The reader is strongly encouraged to do this as an exercise. {\em
It is extremely important to pay attention to the orientation of the
simplices, and keep track of the signs accordingly.}

\vspace{0.2in}

{\footnotesize \hspace{-0.2in}
\begin{tabular}{c|c}
{\bf \large Spaces} & {\bf \large Cochains}  \\
\hline

$f \: X \to Y$ continuous map &  $f \: B \to C$ cochain map
\\
\hline
\begin{tabular}{c} \\
 $\Cyl(f)=(X\times I)\cup_f Y$ \\
\begin{picture}(150,110)(0,-4)
\includegraphics{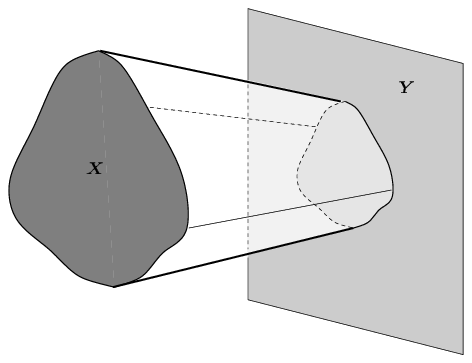}
\end{picture}

\end{tabular}
 &
  \begin{tabular}{c}      {\bf Mapping cylinder} \\ \\

                            $\Cyl(f)^n:=B^n\oplus B^{n+1} \oplus C^n$  \\
                            $B^n\oplus B^{n+1} \oplus C^n \llra{d}
                                 B^{n+1}\oplus B^{n+2} \oplus C^{n+1}$ \\
                            $(b',b,c) \mapsto
                             \big(d_B(b')-b,-d_B(b),f(b)+d_C(c)\big)$
                             \\ \\

                                 $\left(
                                   \begin{array}{ccc}
                                     d_B & -\id_B & 0 \\
                                     0   & -d_B   & 0 \\
                                     0   & f    & d_C \\
                                   \end{array}
                                               \right) $
                                                      \end{tabular}\\
\hline
\begin{tabular}{c} \\
$\Cone(f)=\Cyl(f)/X$ \\
\begin{picture}(150,100)(0,-5)\includegraphics{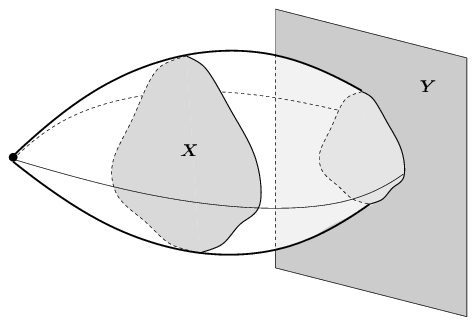}
\end{picture} \end{tabular}
&
    \begin{tabular}{c}   {\bf Mapping cone} \\ \\  \\ $\Cone(f)=\Cyl(f)/B$ \\ \\
       $\Cone(f)^n:=B^{n+1}
                                                             \oplus   C^n$      \\
                                $B^{n+1}\oplus C^n\llra{d}B^{n+2}\oplus C^{n+1}$\\
                                $(b,c) \mapsto
                                \big(-d_B(b),f(b)+d_C(c)\big)$

                                                      \end{tabular}\\
\hline
\begin{tabular}{c} \\
 $\Cone(X)=\Cone(\id_X)=X\times I / X\times \{0\}$ \\
\begin{picture}(150,85)(0,-5)\includegraphics{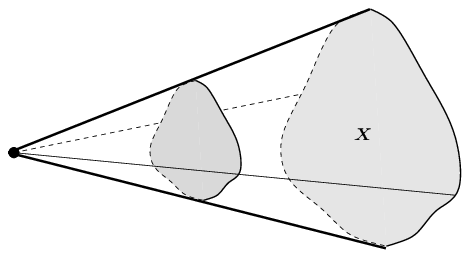}
\end{picture} \end{tabular}
&
  \begin{tabular}{c}    {\bf Cone} \\ \\
                              $\Cone(B):=\Cone(\id_B)$\\ \\
                                         $\Cone(B)^n=B^{n+1}\oplus B^n$\\
                                         $B^{n+1}\oplus B^n\llra{d}
                                                 B^{n+2}\oplus B^{n+1}$\\
                                         $(b,b')\mapsto
                                         \big(-d(b),b+d(b')\big)$
                                           \end{tabular}\\
\hline
\begin{tabular}{c} \\
 $\Sigma(X)=\Cone(X)/X=X\times I / X\times \{0,1\}$ \\
\begin{picture}(150,80)(0,-7)\includegraphics{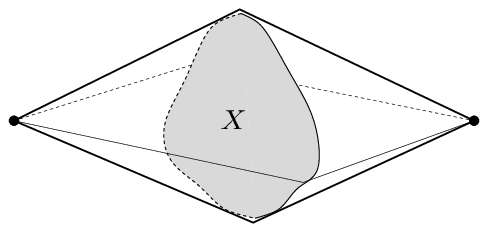}
\end{picture} \end{tabular}
 &  \begin{tabular}{ c}     {\bf Shift (Suspension)} \\ \\
                                     $B[1]:=\Cone(B)/B$      \\ \\
                                             $B[1]^n:=B^{n+1}$        \\
                                          $B[1]^n\llra{d}   B[1]^{n+1}$\\
                                               $b\mapsto -d(b)$       \\
                                       \end{tabular}\\
\hline
\end{tabular}
}

\vspace{0.1in}

\noindent {\em Some history.} {\small These constructions were
originally used by topologists (e.g., Puppe \cite{Puppe}) as a
unified way of treating various cohomology theories for topological
spaces (the so-called {\em generalized cohomology theories}
\cite{Adams}), and soon they became a standard tool in {\em stable
homotopy theory}. Around the same time, Grothendieck and Verdier
developed the same machinery for cochain complexes, to be used in
Grothendieck's formulation of duality theory and the Riemann-Roch
theorem for schemes \cite{Hartshorne2}. This led to the invention of
{\em derived categories} and {\em triangulated categories}
\cite{Verdier}. As advocated by Grothendieck, it has become more and
more apparent over time that working in {\em derived categories} is
a better alternative to working with cohomology. Nowadays people
consider even richer structures such as {\em dg-categories}
\cite{BoKa,Keller3}, $A_{\infty}$-{\em categories} \cite{Keller4},
{\em (stable) model categories} (\cite{Ho}, Section 7), etc. The
idea is that, the higher one goes in this hierarchy, the more {\em
higher-order} cohomological information (e.g.,  Massey products,
Steenrod operations, or other cohomology operations) is retained. In
these lectures we will not go beyond derived categories.
 }

\separate

\noindent {\bf An important fact.} Taking the mapping cone of $f\: X
\to Y$ is somewhat like taking cokernels: we are essentially killing
the image of $f$, but softly (i.e., we make it contractible). That
is why $\Cone(f)$ is sometimes called the {\em homotopy cofiber} or
{\em homotopy cokernel} of $f$; see \cite{May1}, $\S$8.4,
especially, the first lemma on page 58.

\begin{prop}
   In the following diagram of cochain complexes, if $g$ and $h$ are
   chain
   equivalences, then so is the induced map on the mapping cones
     $$\xymatrix@=14pt@M=6pt@C=10pt{B \ar[r]^f \ar[d]^g  &
                           C\ar[r] \ar[d]^h & \Cone(f) \ar[d] \\
        B' \ar[r]^{f'}   & C' \ar[r] & \Cone(f')   }$$
   A similar statement is true for topological spaces.
\end{prop}

\begin{rem}
  The above proposition is still true if we use {\em
  quasi-isomorphisms}  instead of chain homotopy equivalences.
  Similarly, the topological version remains valid if we use
  {\em weak equivalences}
  (i.e., maps that induce isomorphisms on all homotopy groups)
  instead of homotopy equivalences.
\end{rem}

\noindent{\em Exercise.} Give an example (in both the topological
and chain complex settings) to show that the above proposition is
not true if we use the strict cofiber $Y/f(X)$ instead of
$\Cone(f)$. (Therefore, since we are interested in chain  homotopy
equivalence classes, or quasi-isomorphism classes, of cochain
complexes, $\Cone(f)$ is a better-behaved notion than $Y/f(X)$.)

\vspace{0.1in}

\noindent{\em Exercise.}
  Formulate a universal property for the mapping cone construction.
\vspace{0.1in}

\section{Basic properties of cofiber sequences}

\noindent {\small References: \cite{Weibel}, \cite{GeMa},
\cite{KaSch}, \cite{Hartshorne2},  \cite{Verdier}.}

\vspace{0.1in}

\noindent The sequence
  \begin{center}
\fbox{
\begin{minipage}{1.4in}
            $B \llra{f} C \lra \Cone(f)$
\end{minipage}}
\end{center}
\vspace{0.1in}
 (or any sequence quasi-isomorphic to such a sequence)
is called a {\bf cofiber sequence}. The same definition can be made
with topological spaces. We list the basic properties of cofiber
sequences.

\separate

\noindent{\bf 1. Exact sequence, basic form.} An important property
of cofiber sequences is that they give rise to long exact sequences.
Here is a baby version which is very easy to prove and is left as an
exercise.

\begin{prop}{\label{P:exact}}
 Let $B \to C \to \Cone(f)$ be a cofiber sequence. Then the sequence
    $$ H^n(B) \to H^n(C) \to H^n(\Cone(f))$$
 is exact for every $n$.
\end{prop}

\begin{rem}{\label{R:cofiber}}
  This is of course part of a long exact sequence
   {\footnotesize
  $$\cdots \overset{\partial}{\to} H^n(B) \to H^n(C) \to H^n(\Cone(f))
          \overset{\partial}{\to} H^{n+1}(B) \to H^{n+1}(C) \to H^{n+1}(\Cone(f))
                                            \overset{\partial}{\to}
                                            \cdots$$}
  that we will discuss shortly.
  If we use the fact that a short exact sequence of cochain
  complexes gives rise to a long exact sequence of cohomology
  groups, this is easily proven by showing that the cofiber sequence
  $B \to C \to \Cone(f)$ is chain homotopy equivalent to the short
  exact sequence $B \to \Cyl(f) \to \Cone(f)$. The equivalence
  $C\sim\Cyl(f)$ is given by
 \vspace{-0.1in}
\begin{center}
   \begin{tabular}{ccc} $C$ & $\llra{}$ & $\Cyl(f)$
     \\ $c$ & $\mapsto$ &  $(0,0,c)$
   \end{tabular} \ \ \ \ and  \ \ \ \
   \begin{tabular}{ccc} \\  $\Cyl(f)$ & $\llra{}$ & $C$ \\
               $B^n\oplus B^{n+1}\oplus C^n$ & $\lra$ & $C^n$
                     \\ $(b',b,c)$ & $\mapsto$ & $f(b')+c$.
        \end{tabular}
\end{center}
\end{rem}

{\small
\begin{rem}
   We have a similar statement
   for a cofiber sequence $X \to Y \to \Cone(f)$ of topological
   spaces. The corresponding long exact sequence in (co)homology should then
   be interpreted as the long exact sequence of {\em relative}
   (co)homology groups. More specifically,
   if $f$ is an inclusion, we have
     $$H_*(\Cone(f))\cong H_*(Y,X) \ \ \text{and} \ \ H^*(\Cone(f))\cong H^*(Y,X),$$
   where the right hand sides are relative (co)homology groups.
\end{rem}
}

\separate

\noindent {\bf 2. A short exact sequence is a cofiber sequence.} We
pointed out in Remark \ref{R:cofiber} that a cofiber sequence is
chain equivalent (hence quasi-isomorphic) to a pointwise split short
exact sequence. The converse is also true.

\begin{prop}
Consider a short exact sequence of cochain complexes
  $$0 \to B \torel{f} C \to C/B \to 0.$$
Then the natural chain map $\varphi \: \Cone(f) \to C/B$ defined by
  \begin{center}
       \begin{tabular}{ccc}  $B^{n+1}\oplus C^n$ & $\to$ & $C^n/B^n$ \\
                    $(b,c)$ & $\mapsto$ & $\bar{c}$
       \end{tabular}
  \end{center}
 is a quasi-isomorphism. 
 If the sequence is pointwise split,
  with $s \: C^n/B^n \to C^n$ a choice of splitting,
  then $\varphi$ is a  chain equivalence. The inverse is given
  by $\psi \: C/B \to \Cone(f)$
     \begin{center}
      \begin{tabular}{ccc}  $C^n/B^n$ & $\to$ & $B^{n+1}\oplus C^n$ \\
                    $x$ & $\mapsto$ & $\big(sd(x)-ds(x),s(x)\big).$
       \end{tabular}
     \end{center}
\end{prop}

\noindent {\em Exercise.} Let $f \: B \hookrightarrow C$ be an
inclusion of $R$-modules, viewed as cochain complexes concentrated
in degree 0. Show that $\varphi \: \Cone(f) \to C/B$ defined above
is a chain equivalence if and only if $f$ is split.

\separate

\noindent {\bf 3. Iterate of a cofiber sequence.}
  Consider a cofiber sequence
     $$B \llra{f} C \llra{g} \Cone(f).$$
   Observe that $g \: C \to \Cone(f)$ is a pointwise split inclusion.
   This implies that the mapping cone of $g$
   is naturally chain equivalent to $\Cone(f)/C=B[1]$. More
   precisely, we have the following commutative diagram:
    $$\xymatrix@=14pt@M=6pt@C=14pt{B \ar[r]^f   &
                               C \ar@{^(->}[r]^(0.35)g  \ar@{=}[d]
                            & \Cone(f) \ar@{=}[d] \ar[r]^{\partial} &  B[1]\\
      &   C \ar@{^(->}[r]^(0.35)g   & \Cone(f) \ar@{^(->}[r]^{h} \ar[r]
                 & \Cone(g). \ar[u]_(0.48){\text{ch. eq}}^(0.46){\varphi}  }$$
   Note that we have natural chain equivalences
   $$\Cone(\partial)\sim \Cone(h) \sim \coker(h)=C[1].$$
   Indeed, a careful chasing through the above chain equivalences (for which
   we have given explicit formulas) shows that
   the sequence
      $$\Cone(f) \llra{\partial} B[1] \llra{-f[1]} C[1]$$
   is a cofiber sequence. We can now iterate this process forever
   and produce a long sequence
\vspace{0.1in}
  \begin{center}
\fbox{
\begin{minipage}{4.5in}
             $\cdots \llra{\partial[-1]} B \llra{f} C \llra{g} \Cone(f)
          \llra{\partial} B[1] \llra{-f[1]} C[1] \llra{-g[1]} \Cone(f)[1]
                                            \llra{-\partial[1]}
                                            \cdots$
\end{minipage}} $(\bigstar)$
\end{center}
\vspace{0.1in} in which every three consecutive terms form a cofiber
sequence. \label{sequence}

{\small
\begin{rem}
  A similar long exact sequence can be constructed for a map $f \: X
  \to Y$ of topological spaces, and it is called a {\em Puppe sequence}
  (or {\em cofiber sequence} \cite{May1} $\S$8.4).
  A Puppe sequence,
  however, only extends to the right. The reason for this is that
  the suspension functor on the category of topological spaces is not
  invertible (as opposed to the shift functor for cochain
  complexes).
  The Puppe sequence can be used, among other things, to
  give a natural construction of the  long exact (co)homology sequence of a pair.
\end{rem}
}

\separate

\noindent {\bf 4. Long exact sequence of a cofiber sequence.}
  Applying Proposition \ref{P:exact} to the above long exact
  sequence of chain complexes, and observing that $H^i(B[n])=H^{i+n}(B)$,
  we obtain a long exact sequence
  of cohomology groups

  {\footnotesize
    $$\cdots \overset{\partial}{\to} H^n(B) \to H^n(C) \to H^n(\Cone(f))
          \overset{\partial}{\to} H^{n+1}(B) \to H^{n+1}(C) \to H^{n+1}(\Cone(f))
                                            \overset{\partial}{\to}
                                            \cdots.$$}
\separate

\noindent {\bf Moral.}
   By exploiting the notion of cone of a map in a systematic way, we can
   elevate many basic cohomological constructions to the level
   of chain complexes. For this to work conveniently, one needs to
   {\em invert} chain equivalences  so that
   one can treat such morphisms as isomorphisms. Indeed, for various reasons,
   inverting
   only chain equivalences is usually not enough.
   For example, recall that one of our motivations for working with chain
   complexes was that we wanted to be able to replace an
   object $M \in \A$ with a better behaved resolution (projective
   or injective) of it. Since the resolution
   map is only a quasi-isomorphism and not a chain equivalence in
   general,   it is only after inverting quasi-isomorphisms that
    we can regard an $M \in \A$ and its resolution as the ``same''.

\section{Derived categories}

\noindent {\small References: \cite{Weibel}, \cite{GeMa},
\cite{Hartshorne2},  \cite{KaSch}, \cite{Iversen}, \cite{Verdier},
\cite{Keller1}, \cite{Keller2}.}

\vspace{0.1in}

\noindent Let $\A$ be an abelian category.

\vspace{0.1in}

\noindent The {\bf homotopy category} $\mathcal{K}(\A)$ is the
category obtained by inverting all chain equivalences in the
category $\Ch(\A)$ of chain complexes. More precisely, there is a
natural functor $\Ch(\A) \to \mathcal{K}(\A)$ which has the
following universal property:

\vspace{0.1in}

 \fbox{
 \begin{tabular}{cr}
  \begin{tabular}{l} \\
 If a functor $F \: \Ch(\A) \to \C$  sends chain \\ equivalences to
 isomorphisms, then
  \end{tabular} &
     $\xymatrix@=16pt@M=8pt@C=28pt{ \Ch(\A) \ar[r]^F \ar[d] & \C \\
          \mathcal{K}(\A)    \ar@{..>}_{\exists !}[ur] & }$
  \end{tabular}
      }

\vspace{0.1in}

\noindent  If in the above definition we replace `chain equivalence'
by `quasi-isomorphism', we arrive at the definition of the
 {\bf derived category} $\mathcal{D}(\A)$. We list the
basic properties of $\mathcal{K}(\A)$ and $\mathcal{D}(\A)$.

\separate

\noindent{\bf 1. Explicit construction of $\mathcal{K}(\A)$.} We
can, alternatively,
  define the homotopy category $\mathcal{K}(\A)$ to be the category
  whose objects are the ones of $\Ch(\A)$ and whose morphisms are
  defined by

\begin{center}
  \fbox{
    $\Hom_{\mathcal{K}(\A)}(B,C):=\Hom_{\Ch(\A)}(B,C)/N$
   }\end{center}
\vspace{0.1in}
 where $N$ stands for the group of null homotopic maps. ({\em Exercise.} Show
 that this category satisfies the required
 universal property.)
 This, in particular, implies that $\mathcal{K}(\A)$ is an additive
 category.

\separate

\noindent {\bf 2. Cohomology functors}.
  The cohomology functors $H^n \: \Ch(\A) \to \A$ factor
  through $\mathcal{K}(\A)$ and $\mathcal{D}(\A)$:
     $$\xymatrix@=16pt@M=8pt@C=28pt{ \Ch(\A) \ar[rrd]_{H^n} \ar[r]&
         \mathcal{K}(\A) \ar[rd]^{H^n} \ar[r] & \mathcal{D}(\A) \ar[d]^{H^n} \\
         &  & \A }.$$

\separate

\noindent {\bf 3. Explicit construction of $\mathcal{D}(\A)$.}
   The objects of $\mathcal{D}(\A)$ can again be taken to be
   the ones of $\Ch(\A)$.
   Description of morphisms in $\mathcal{D}(\A)$ is, however, slightly
   more
   involved than the case of $\mathcal{K}(\A)$. It requires the notion of
   {\em calculus of fractions} in
   a category, which we will not get into for the sake of brevity.

     Just to give an idea how this works, first we observe that
   $\mathcal{D}(\A)$
   can, equivalently, be defined by inverting quasi-isomorphisms in
   $\mathcal{K}(\A)$. The class of quasi-isomorphisms in $\mathcal{K}(\A)$
   satisfies the axioms of the calculus of fractions (\cite{GeMa}, Definition III.2.6),
   and this allows one to give
   an explicit description of morphisms in $\mathcal{D}(\A)$
   (see \cite{GeMa}, Section III.2, \cite{Hartshorne2}, Section I.3, or
   \cite{Weibel}, Section 10.3). A consequence of this is that,
   for every morphism $f \: B \to C$ in
   $\mathcal{D}(\A)$, one can find morphisms $s$, $t$, $g$ and $h$
   in $\mathcal{K}(\A)$ such that $s$ and $t$ are quasi-isomorphisms
   and the following diagrams commute in
   $\mathcal{D}(\A)$:
      $$\xymatrix@=16pt@M=8pt@C=28pt{ A \ar[d]_s \ar[rd]^g &  & &
      A'
      \\
        B \ar[r]_f& C  & B \ar[ru]^h \ar[r]_f & C \ar[u]_t}$$
Also,  for any two morphisms $f,f' \: B \to C$, now in
$\mathcal{K}(\A)$, which become equal in $\mathcal{D}(\A)$, there
are $s$ and $t$ as above such that $f\circ s=f'\circ s$ and $t\circ
f=t\circ f'$.

\separate

\noindent {\bf 4. Cofiber sequences.}
   There is a notion of cofiber sequence in both $\mathcal{K}(\A)$
   and $\mathcal{D}(\A)$. Every morphism $f \: B \to C$ fits in
   a sequence $\bigstar$ as in page \pageref{sequence} in which every three
   consecutive terms form a cofiber sequence.

\section{Variations on the theme of derived categories}

\noindent {\small References: \cite{Weibel}, \cite{GeMa},
\cite{Hartshorne2},  \cite{KaSch}, \cite{Iversen}, \cite{Keller1},
\cite{Keller2}.}

\vspace{0.1in}

To an abelian category $\A$ we can associate various types  of
homotopy or derived categories by imposing certain boundedness
conditions on the chain complexes in question.

\separate

\noindent The {\bf bounded below} derived category
$\mathcal{D}^+(\A)$ is the category obtained from  inverting the
quasi-isomorphisms in the category of bounded below chain complexes
$\Ch^+(\A)$. Equivalently, $\mathcal{D}^+(\A)$ is obtained from
inverting the quasi-isomorphisms in the homotopy category
$\mathcal{K}^+(\A)$ of $\Ch^+(\A)$.

\vspace{0.1in}

\noindent In the same way one  can define {\bf bounded} derived
categories $\mathcal{D}^-(\A)$
 and $\mathcal{D}^b(\A)$, where $b$ stands for bounded on two sides.

\begin{rem}
  One can use the calculus of fractions in $\mathcal{K}^+(\A)$
  (respectively, $\mathcal{K}^-(\A)$, $\mathcal{K}^b(\A)$)
  to give a description of morphisms in $\mathcal{D}^+(\A)$
  (respectively, $\mathcal{D}^-(\A)$, $\mathcal{D}^b(\A)$).
  It follows that each of these bounded derived categories can be
  identified with a full subcategory of  $\mathcal{D}(\A)$.
  In particular, $\mathcal{D}^b(\A)=\mathcal{D}^+(\A)\cap \mathcal{D}^-(\A)$.
  We will say more on this in the next lecture.
\end{rem}

\separate

There is an alternative way of computing morphisms in {\em bounded}
derived categories using injective or projective resolutions. It is
summarized in the following theorems.

\vspace{0.1in}

\begin{thm}[\oldcite{GeMa},  $\S$ III.5.21, \oldcite{Weibel},
Theorem 10.4.8]{\label{T:derive}}
  Let $\mathcal{I}^+(\A) \subset \mathcal{K}^+(\A)$ be the full
  subcategory consisting of complexes whose terms are all injective.
  Then the composition
     $$\xymatrix@=16pt@M=8pt@C=28pt{ \mathcal{I}^+(\A) \ar@{^(->}[r]
       \ar@/_1pc/[rr] &
       \mathcal{K}^+(\A) \ar[r]  & \mathcal{D}^+(\A)}$$
  is fully faithful. It is an equivalence if $\A$ has enough injectives.
  The same thing is true if we replace
  $+$ by $-$ and injective by projective:
     $$\xymatrix@=16pt@M=8pt@C=28pt{ \mathcal{P}^-(\A) \ar@{^(->}[r]
       \ar@/_1pc/[rr] &
       \mathcal{K}^-(\A) \ar[r]  & \mathcal{D}^-(\A)}.$$
\end{thm}

 The (first part of the) above theorem says that, if $I$ and
$J$ are bounded below complexes of injective objects, then we have
\vspace{0.1in}
     \begin{center}  \fbox{
  $\Hom_{\mathcal{D}(\A)}(J,I)=\Hom_{\mathcal{K}(\A)}(J,I).$}\end{center}
\vspace{0.1in} ({\em This is actually true if $J$ is an arbitrary
bounded below complex}; see \cite{Weibel}, Corollary 10.4.7). In
particular,  if $I$ and $J$ are quasi-isomorphic, then they are
chain equivalent. This is good news, because we have seen that
computing $\Hom$ is much easier in $\mathcal{K}(\A)$. Now if $\A$
has enough injectives, it can be shown that for every bounded $C$,
there exists a bounded below complex of injectives $I$, and a
quasi-isomorphism $C \risom I$. This is called an {\em injective
resolution} for $C$. So, by virtue of the above fact, injective
resolutions can be used to compute morphisms in derived categories:
\vspace{0.1in}
 \begin{center}  \fbox{
  $\Hom_{\mathcal{D}(\A)}(B,C)=\Hom_{\mathcal{K}(\A)}(B,I).$}
 \end{center}

\vspace{0.1in}

\noindent{\em Exercise.} Set
 $\Hom^i_{\mathcal{D}(\A)}(B,C):=\Hom_{}(B,C[i])$. Let $B$ and $C$
 be objects in $\A$, viewed as complexes concentrated in degree 0.
 Assume $\A$ has enough injectives. Show that
\vspace{0.1in}
  \begin{center}  \fbox{
     $\Hom^i_{\mathcal{D}(\A)}(B,C)\cong\Ext^i(B,C).$}
  \end{center}
\vspace{0.1in}
 Use the second part of the theorem to give a way of computing
 $\Ext$ groups using a projective resolution for $C$, if they exist. Compute
 $\Hom^i_{\mathcal{K}(\A)}(B,C)$ and compare it with
 $\Hom^i_{\mathcal{D}(\A)}(B,C)$.

\vspace{0.1in}

\noindent{\em Exercise.} Give an example of a non-zero morphism in
$\mathcal{D}(\A)$, with $\A$ your favorite abelian category, which
induces zero maps on all cohomologies.

\section{Derived functors}

\noindent {\small References: \cite{Weibel}, \cite{GeMa},
\cite{Hartshorne2},  \cite{KaSch}, \cite{Iversen}, \cite{Keller1},
\cite{Keller2}.}

\vspace{0.1in}

To keep the lecture short, we will skip the very important topic of
derived functors and confine ourselves to an example: the {\em
derived tensor} $\overset{L}{\otimes}$. The reader is encouraged to
consult the given references for the general discussion of derived
functors, especially the all important {\em derived hom}
$\operatorname{RHom}$.

\separate

Let $\A=R$-$\mathbf{Mod}$. The idea is that we want to have a notion
of tensor product for chain complexes which is well defined on
$\mathcal{D}(A)$.

\vspace{0.1in}

 \noindent {\em The usual tensor product of chain complexes}
(Definition \cite{Weibel}, $\S$2.7.1) {\em does not pass to derived
categories.} This is because tensor product is not exact. More
precisely, if $A \to A'$ is a quasi-isomorphism, then $A\otimes B
\to A'\otimes B$ may no longer be. (For a counterexample, let $A$ be
a short exact sequence, $A'=0$, and $B$ a complex concentrated in
degree 0.)

\vspace{0.1in}

\noindent {\em The usual tensor product of chain complexes DOES pass
to homotopy categories.} This is because a null homotopy $s$ for a
chain map $A \to A'$ gives rise to a null homotopy $s\otimes B$ for
$A\otimes B \to A'\otimes B$ (exercise).

\vspace{0.1in}

In particular, taking all complexes to be (bounded-above) complexes
of projective modules, we get a well-defined tensor product
$-\otimes-\: \mathcal{P}^-(\A)\times\mathcal{P}^-(\A) \to
\mathcal{P}^-(\A)$. Since $R$-$\mathbf{Mod}$ has enough projectives,
we obtain, via the equivalence $\mathcal{P}^-(\A)\cong
\mathcal{D}^-(\A)$ of Theorem \ref{T:derive}, the desired tensor
product on the bounded above derived category: 
\begin{center}\fbox{
 $-\overset{L}{\otimes}-\:\mathcal{D}^-(\A)\times\mathcal{D}^-(\A)
                \to \mathcal{D}^-(\A).$}\end{center}
 \noindent More explicitly, $A\overset{L}{\otimes}B$
is defined to be $P\otimes Q$, where $P\to A$ and $Q\to B$ are
certain chosen projective resolutions. (In fact, it is enough to
resolve only one of $A$ or $B$.)

\vspace{0.1in}

\noindent {\em Exercise.} Show that if $M, N \in \A$ are viewed as
complexes concentrated in degree $0$, then
\begin{center}\fbox{
$H^{-i}(M\overset{L}{\otimes}N)\cong
\operatorname{Tor}_{i}(M,N).$}\end{center}

\newpage
{\LARGE\part*{Lecture 3: Triangulated categories}}

\vspace{0.3in}

\noindent{\bf Overview.} The main properties of the cone
construction for chain complexes can be formalized into a set of
axioms. This leads to the notion of a {\em triangulated} category,
the main topic of this lecture. A triangulated category is an
additive category in which there is an abstract notion of mapping
cone.  The cohomology functors on chain complexes can also be
studied at an abstract level in a triangulated category.

The main example of a triangulated category is the derived category
$\D(\A)$ of chain complexes in an abelian category $\A$. Using a
{\em $t$-structure} on a triangulated category, we can produce an
abelian category, called the {\em heart} of the $t$-structure. For
example, there is a standard $t$-structure on $\D(\A)$ whose heart
is $\A$. The $t$-structure is not unique, and by varying it we can
produce new abelian categories. We show how using a {\em torsion
theory} on an abelian category $\A$ we can produce a new
$t$-structure on $\D(\A)$. The heart of this $t$-structure is then a
new abelian category $\B$. For example, this method has been used in
noncommutative geometry to ``deform'' the abelian category $\A$ of
coherent sheaves on a torus $X$. The new abelian category $\B$ can
then be thought of as the category of coherent sheaves on a
``noncommutative deformation'' of $X$, a noncommutative torus.

\setcounter{section}{0}

\vspace{0.3in}

\section{Triangulated categories}

\noindent {\small References: \cite{Weibel}, \cite{GeMa},
\cite{Hartshorne2}, \cite{Neeman}, \cite{BeBeDe}, \cite{KaSch},
\cite{Iversen}, \cite{Verdier}, \cite{Keller1}, \oldcite{Keller2}.}

\vspace{0.1in}

We formalize the main properties of the mapping cone construction
and define triangulated categories.

\separate

\noindent Let $\T$ be an additive category  equipped with an
auto-equivalence $X \mapsto X[1]$  called {\bf shift} (or {\bf
translation}).

\vspace{0.1in}

 \noindent By a
{\bf triangle} in $\T$ we mean a sequence $X \to Y \to Z \to X[1]$.
We sometimes write this as
$$\xymatrix@R=8pt@C=2pt@M=5pt{ & Y \ar[rd]^g & \\
  X \ar[ru]^f & & Z \ar[ll]^{+1}}$$

\noindent  A {\bf triangulation} on $\T$ is a collection of
triangles, called {\bf exact} (or {\bf distinguished}) triangles,
satisfying the following axioms:

\vspace{0.1in}
  \begin{itemize}
     \item[$\mathbf{TR1.}$]
              \begin{itemize}
                 \item[a)]  $X \llra{\id} X \lra 0 \lra
                    X [1]$ is an exact triangle.
                 \item[b)] Any morphism $f \: X \to Y$ is part of an
                     exact triangle.
                 \item[c)] Any triangle isomorphic to an exact
                     triangle is exact.
              \end{itemize}

        \vspace{0.1in}

     \item[$\mathbf{TR2.}$] \  $X \torel{f} Y \torel{g} Z \torel{h}
         X[1] \ \ \text{exact}  \ \ \Leftrightarrow \ \ Y \torel{g}
                         Z \torel{h} X[1] \torel{-f[1]} Y[1] \ \
         \text{exact}$.

        \vspace{0.1in}

     \item[$\mathbf{TR3.}$] \ In the diagram

    \hspace{0.5in} $\xymatrix@R=16pt@C=20pt@M=7pt{X \ar[r]
                                      & Y \ar[r] & Z \ar[r] & X[1]  \\
          X' \ar[r] \ar[u]^u  & Y' \ar[r] \ar[u]^v  & Z' \ar[r]
                    \ar@{..>}[u]^{\exists} & X'[1] \ar[u]_{u[1]}  }$

  \noindent  if the horizontal rows are exact, and the left square commutes,
  then the dotted arrow can be  filled  (not necessarily uniquely) to make
  the diagram commute.

        \vspace{0.1in}

     \item[$\mathbf{TR4}$] ({\em short imprecise version}). Given $f \: X
     \to Y$ and $g \: Y \to Z$, we have the following commutative
     diagram in which every pair of same color arrows is part of an exact triangle
   $$\xymatrix@R=4pt@C=6pt@M=6pt{& & & & & X \thickar{}{ld}{_{f}} {\ar[ld]}
                                                       \dentar{}{dd}{^gf} & & \\
    & & & &         Y \thickar{}{llllddd}{}  \ar[rd]_g  & &  & \\
    & & & & & Z  \dentar{}{dd}{} \ar[rrdd]  && \\
    & & & & & & &   \\
    U \ar@{..>} [rrrrr] & & & & & V \ar@{..>} [rr] & & W                 }$$
Idea: ``$(Z/X)/(Y/X)=Z/Y)$''.
  See Appendix II for a more precise statement.
  \end{itemize}

\vspace{0.1in}

\noindent A {\bf triangulated category} is  an additive category
$\mathcal{T}$ equipped with an auto-equivalence $X\mapsto X[1]$  and
a triangulation.

\vspace{0.1in}

\noindent   A {\bf triangle functor} (or a {\bf exact
functor})\footnote{This is not a very good terminology, because
there is also a notion of {\em $t$-exact} functor with respect to a
$t$-structure.} is a functor
 between triangulated categories which commutes with
 the translation functors (up to natural transformation)
 and sends  exact triangles to exact triangles.

\begin{rem}
  A triangulated category can be thought of as a category in which there
  are well-behaved notions of {\em homotopy kernel} and {\em homotopy
  cokernel}\footnote{Indeed not quite well-behaved (read, functorial),
  essentially due to the fact that the morphism whose existence is required
  in $\mathbf{TR3}$ may not be unique. This turns out to be problematic.
   A way to remedy this is to work with
  DG categories; see \oldcite{BoKa}.} (hence also various other types of homotopy
  limits and colimits).
\end{rem}

\separate

\noindent {\bf Some immediate consequences.} Let $\mathcal{T}$ be a
triangulated category, and let $X \to Y \to Z \to X[1]$  be an exact
triangle in $\mathcal{T}$. Then, the following are true:

  \vspace{0.1in}

\begin{itemize}
    \item[$\bf{0.}$] The opposite category $\T^{op}$ is naturally triangulated.
    \vspace{0.1in}

    \item[$\bf{1.}$] Any length three portion of the sequence
     {\footnotesize
       $$\cdots \llra{} Z[-1] \llra{-h[-1]} X \llra{f} Y \llra{g} Z
        \llra{h} X[1] \llra{-f[1]} Y[1] \llra{-g[1]} Z[1] \llra{-h[1]} \cdots
        $$}
       is an exact triangle.
     \vspace{0.1in}

     \item[$\bf{2.}$] In the above exact sequence, any two
     consecutive morphisms compose to zero.
     {\em Proof.} Enough to check $g\circ f=0$. Apply
     $\mathbf{TR1.}a$ and $\mathbf{TR3}$ to

        \hspace{0.5in} $\xymatrix@R=16pt@C=20pt@M=7pt{X \ar[r]^f
                                      & Y \ar[r]^g & Z \ar[r]^h & X[1]  \\
          X \ar[r]^{\id} \ar[u]^{\id}  & X \ar[r] \ar[u]^f  & 0 \ar[r]
                    \ar@{..>}[u]^{\exists} & X[1] \ar[u]_{\id}  }$

    \vspace{0.1in}

     \item[$\bf{3.}$] Let $T$ be any object. Then,
        $$\Hom(T,X) \to \Hom(T,Y) \to \Hom(T,Z)$$
        is exact. Therefore, the long sequence of ($\mathbf{1}$)
        gives rise to a long exact sequence of abelian groups. The same
        thing is true for
          $$\Hom(X,T) \to \Hom(Y,T) \to \Hom(Z,T).$$
\end{itemize}

\vspace{0.1in}

{\small
\begin{rem}
   In fact, $\mathbf{TR3}$ and half of $\mathbf{TR2}$ follow from
   the rest of the axioms. For this, see \cite{May2}.
   Also see \cite{Neeman} for another formulation of $\mathbf{TR4}$;
   especially,  Remark 1.3.15 and Proposition 1.4.6.
\end{rem}
}

\separate

\noindent {\bf The main example.}
 The  proof of the following theorem can be found in
 \cite{GeMa} or \cite{KaSch}. (We have sketched some
 of the ideas in Lecture 2.)

\begin{thm}
  Let $\A$ be an abelian category. Then $\mathcal{K}(\A)$ and
  $\mathcal{D}(\A)$ are triangulated categories.
\end{thm}

 In both cases, the distinguished triangles are the ones obtained
from cofiber sequences. Equivalently, the distinguished triangles
are the ones obtained from the pointwise split short exact
sequences. (To see the equivalence, note that one implication
follows from Proposition 4.4 of Lecture 2. The other follows from
the fact that the cofiber sequence $B \to C \to \Cone(f)$ is chain
homotopy equivalent to the pointwise split short exact sequence $B
\to \Cyl(f) \to \Cone(f)$; see Remark 4.2 of Lecture 2.) In the case
of $\mathcal{D}(\A)$ we get the same triangulation if we take {\em
all} short exact sequences. This also follows from Proposition 4.4
of Lecture 2.

\separate

  The following proposition allows us to produce more triangulated
  categories from the above basic ones.

\begin{prop}
  Let $\A$ be an abelian category, and let $\mathcal{C}$ be a full
  additive subcategory of $\Ch(\A)$. Let $\mathcal{K} \subset
  \mathcal{K}(\A)$ be the corresponding quotient category and
  $\mathcal{D}$ the localization of $\mathcal{K}$ with respect to
  quasi-isomorphisms. Assume $\mathcal{C}$ is closed
  under translation, quasi-isomorphisms, and forming mapping cones.
  Then $\mathcal{K}$
  and $\mathcal{D}$ are triangulated categories and we have fully faithful
  triangle functors
     $$\mathcal{K} \to \mathcal{K}(\A)$$
     $$\mathcal{D} \to \mathcal{D}(\A).$$
\end{prop}

\begin{proof}
  The proof  in the case of $\mathcal{K}$ is straightforward.
  The  case of $\mathcal{D} \to \mathcal{D}(\A)$
  follows easily from the existence of calculus of fractions on $\mathcal{K}(\A)$.
  (Our discussion of calculus of fractions in Lecture 2 is enough for this.)
\end{proof}

\vspace{0.1in}

\noindent {\em Exercise.} Show that $\mathcal{D}^-(\A)$ can,
equivalently, be defined by inverting the quasi-isomorphisms in the
category  $\mathcal{C}$ of chain complexes with bounded above
cohomology. The similar statement is true in the case of
$\mathcal{D}^+(\A)$ and $\mathcal{D}^b(\A)$. (Hint: construct a
right inverse to the inclusion $\iota \: \Ch^-(\A) \hookrightarrow
\mathcal{C}$ by choosing an appropriate truncation of each complex
in $\C$; see page \pageref{truncate} to learn how to truncate. Show
that this becomes an actual inverse to $\iota$ when we pass to
derived categories.)

\vspace{0.1in}

\begin{cor} We have the following
fully faithful triangle functors:
    $$\mathcal{K}^-(\A), \mathcal{K}^+(\A), \mathcal{K}^b(\A)
    \  \hookrightarrow \ \mathcal{K}(\A),$$
    $$\mathcal{D}^-(\A), \mathcal{D}^+(\A), \mathcal{D}^b(\A)
    \ \hookrightarrow \ \mathcal{D}(\A).$$
\end{cor}

\begin{proof}
  This is an immediate corollary of the previous proposition.
  (We also need to use the previous exercise.)
\end{proof}

\separate

\noindent{\em Important examples to keep in mind:}

\begin{itemize}
  \item[$\mathbf{1.}$] $\A=R$-$\mathbf{Mod}$.

     \vspace{0.1in}

  \item[$\mathbf{2.}$] $\A$ = sheaves of $R$-modules over a
  topological space.

     \vspace{0.1in}

  \item[$\mathbf{3.}$] $\A$ = presheaves of $R$-modules over a
  topological space.

     \vspace{0.1in}

  \item[$\mathbf{4.}$] $\A$ = sheaves of $\mathcal{O}_X$-modules on a scheme $X$.

    \vspace{0.1in}

  \item[$\mathbf{5.}$] $\A$ = quasi-coherent sheaves on  a scheme $X$.

\end{itemize}

\section{Cohomological functors}{\label{S:Cohomological}}

\noindent {\small References: \cite{Weibel}, \cite{GeMa},
\cite{BeBeDe} \cite{Hartshorne2}, \cite{KaSch}, \cite{Verdier}.}

\vspace{0.1in}

  \noindent Let $\T$ be a triangulated category and $\A$ an abelian
  category. An additive functor $H \: \T \to \A$ is called
  a {\bf cohomological functor} if for every exact triangle
   $X \to Y \to Z \to X[1]$,
   \vspace{0.1in}

\begin{center}
 \fbox{$H(X) \to H(Y) \to H(Z)$ is exact in $\A$.}
\end{center}
 \vspace{0.1in}

\noindent If we set $H^n(X):=H(X[n])$, we obtain the following long
 exact sequence:
  \begin{center}
  \fbox{ {\footnotesize
       $\cdots \llra{} H^{n-1}(Z) \to H^n(X) \to H^n(Y) \to H^n(Z)
        \to H^{n+1}(X) \to H^{n+1}(Y) \to H^{n+1}(Z) \to \cdots
        $} }
   \end{center}

\vspace{0.1in}

\begin{ex}{\label{E:cohomological}}
\end{ex}
\begin{itemize}
    \item[$\bf{1.}$] Let $\T$ be any of
     $\mathcal{K}^*(\A)$ or $\mathcal{D}^*(\A)$, $*=\emptyset,-,+,b$.
     Then the functor $H \: \T \to \A$, $X \mapsto H^0(X)$
     is a cohomological functor.

     \vspace{0.1in}

    \item[$\bf{2.}$] For any $T$, $\Hom(T,-) \: \T \to \A$
      is a cohomological functor. Similarly,
      $\Hom(-,T) \: \T \to \A^{op}$ is a cohomological functor.
\end{itemize}

\vspace{0.1in}

\noindent {\em Exercise.} Show that ($\mathbf{1}$) is a special case
of ($\mathbf{2}$).
\section{Abelian categories inside triangulated categories;
$t$-structures}{\label{S:t}}

\noindent {\small References:  \cite{GeMa}, \cite{KaSch},
\cite{BeBeDe}.}

\vspace{0.1in}

The triangulated categories associated to an abelian category $\A$
contain $\A$ as a full subcategory. More precisely,

\vspace{0.1in}
\begin{prop}
  Let $\T$ be any of
     $\mathcal{K}^*(\A)$ or $\mathcal{D}^*(\A)$, $*=\emptyset,-,+,b$.
     Then the functor $\A \to \T$ defined by
       $$A \ \ \ \ \mapsto \ \ \ \ \cdots \to 0  \to A \to 0  \to
       \cdots$$
      ($A$ is sitting in degree zero)       is fully faithful.
\end{prop}

\vspace{0.1in}

Observe that this is not the only abelian subcategory of $\T$. For
example, we have $\bbZ$ many copies of $\A$ in $\T$ (simply take
shifts of the above embedding). We may also have abelian categories
in $\T$ that are not isomorphic to $\A$.

\separate

As we saw above, the triangulated structure of $\T$ is not
sufficient to recover the abelian subcategory $\A$. What extra
structure do we need on $\T$ in order to reconstruct $\A$?

\vspace{0.1in}

\noindent A {\bf $t$-structure} on a triangulated category $\T$ is a
pair $(\tul{0}, \tug{0})$ of saturated (i.e., closed under
isomorphism) full subcategories such that:

\begin{itemize}

  \item[$\mathbf{t1.}$] If $X \in \tul{0}$, $Y\in \tug{1}$, then
  $\Hom(X,Y)=0$.

  \item[$\mathbf{t2.}$] $\tul{0}\subseteq \tul{1}$, and
  $\tug{1}\subseteq\tug{0}$.

  \item[$\mathbf{t3.}$] For every $X\in\T$, there is an exact
  triangle
    $$A \to X \to B \to A[1]$$
  such that $A \in \tul{0}$ and $B\in \tug{1}$.
\end{itemize}

\vspace{0.1in} \noindent Notation: $\tul{n}=\tul{0}[-n]$ and
$\tug{n}=\tug{0}[-n]$.

\vspace{0.1in} \noindent{\bf Main example.} Let
$\T=\mathcal{D}(\A)$. Set \vspace{0.1in}
  \begin{center}
\fbox{
\begin{minipage}{2.6in}
  $$\tul{0}=\{X\in\mathcal{D}(\A) \ | \ H^i(X)=0, \ \forall i>0\},$$
  $$\tug{0}=\{X\in\mathcal{D}(\A) \ | \ H^i(X)=0, \ \forall i<0\}.$$
\end{minipage}}
\end{center}
\vspace{0.1in}

\noindent The proof that this is a $t$-structure is not hard. The
axioms ($\mathbf{t1}$) and ($\mathbf{t2}$) are straightforward. To
verify ($\mathbf{t3}$), we define {\em truncation functors}
  $$\taudl{0}\: \T \to \tul{0}$$
  $$\taudg{1}\: \T \to \tug{1}$$
by \vspace{0.1in}
  \begin{center}{\label{truncate}}
\fbox{
\begin{minipage}{4.5in}
 $$\taudl{0}(X) := \ \ \ \ \ \ \ \ \ \cdots \to X^{-2} \to X^{-1} \to
  \ker{d} \to 0 \to 0 \to \cdots \    $$
 $$\taudg{1}(X) :=X/\taudl{0}(X)= \ \  \ \  \cdots \to 0 \to 0 \to
  X_0/\ker{d} \to X^1 \to X^2 \to \cdots.$$
 \end{minipage}}
\end{center}
\vspace{0.1in}

\noindent It is easy to see that
  $$\taudl{0}X \to X \to \taudg{1}X \to \taudl{0}X[1]$$
  is exact (Lecture 2, Proposition 4.4). So, we can take
  $A=\taudl{0}X $ and $B=\taudg{1}X$.

\begin{rem}
  This $t$-structure induces $t$-structures on each of
   $\mathcal{D}^{-,+,b}(\A)$.
\end{rem}

\noindent {\em Exercise.} Show that this does not give a
$t$-structure
  on $\mathcal{K}(\A)$. (Hint: show that ($\mathbf{t1}$) fails
  by taking $X=Y$  to be
  an exact complex that is not contractible, e.g., a non-split
  short exact sequence.)

\separate

The truncation functors discussed above can indeed be defined for
any $t$-structure.

\vspace{0.1in}

\begin{prop}[\oldcite{BeBeDe}, $\S$1.3]{\label{P:truncation}}
 Let $\T$ be a triangulated category equipped with a $t$-structure
 $(\tul{0},\tug{0})$.
 \begin{itemize}
    \item[$\mathbf{i.}$] The inclusion $\tul{n}\hookrightarrow \T$
      admits a right adjoint $\taudl{n} \: \T \to \tul{n}$, and the
      inclusion $\tug{n}\hookrightarrow \T$
      admits a left adjoint $\taudg{n} \: \T \to \tug{n}$.

    \item[$\mathbf{ii.}$] For every $X \in \T$ and every $n\in \bbZ$,
     there is a unique $d$ that makes the triangle
       $$\taudl{n}X \lra X \lra \taudg{n+1}X \llra{d}
       \taudl{n}X[1]$$
     exact.
 \end{itemize}

\end{prop}
{\small
\begin{rem}
  An {\em aisle} in a triangulated category is a full saturated subcategory $\U$
  that is closed under extensions, $\U[1]\subseteq \U$, and such that
  $\U \to \T$ admits a right adjoint. For example, $\U=\tul{0}$ is an aisle.
  Conversely, any aisle $\U$ is equal to $\tul{0}$ for a unique $t$-structure
  $(\tul{0},\tug{0})$; see \cite{KeVo}.
\end{rem}
}

\begin{prop}[\oldcite{BeBeDe}, $\S$1.3]
Let $a\leq b$. Then $\taudg{a}$ and $\taudl{b}$ commute, in the
sense that
  $$\xymatrix@C=4pt@R=0pt@M=8pt{
    \taudl{b}X  \ar[dd] \ar[r] & X \ar [r] &  \taudg{a}X \ar[dd] \\
    && \\
       \taudg{a}\taudl{b}X \ar @{..>}^{\exists !}_{\varphi} [rr]&&
                                             \taudl{b}\taudg{a}X  }$$
  is commutative. Furthermore, $\varphi$ is necessarily an isomorphism.
\end{prop}

\noindent We set $\tau_{[a,b]}:=\taudg{a}\taudl{b}X$. We call
$\tul{0}\cap \tug{0}$ the {\bf heart} (or {\bf core}) of the
$t$-structure.

\begin{thm}[\oldcite{BeBeDe}, $\S$1.3]{\label{T:heart}}
   The heart $\C:=\tul{0}\cap \tug{0}$ is an abelian category.
   The functor $H^0:=\tau_{[0,0]} \: \T \to \C$ is a cohomological
   functor.
\end{thm}

\begin{ex}
  The heart of the standard $t$-structure on $\mathcal{D}(\A)$
  is $\A$, and the functor $\tau_{[0,0]}$ is the usual $H^0$.
\end{ex}

\begin{rem}
  In the above theorem, $\T$ may not be equivalent to any
  of $\mathcal{D}^{\emptyset,-,+,b}(\C)$.
\end{rem}

 The moral of the story is that, by varying $t$-structures
 inside a triangulated category (e.g., $\mathcal{D}(\A)$), we can
 produce new abelian categories.

\section{Producing new abelian categories}{\label{S:Tilting}}

\noindent {\small Reference: \cite{HaReSm}.}

\vspace{0.1in}

One can use  $t$-structures to produce new abelian categories out of
given ones. This technique appeared for the first time in
\cite{BeBeDe} where they alter the $t$-structures on various derived
categories of sheaves on a space to produce new abelian categories
of {\em perverse sheaves} (see {\em loc.~cit.} $\S$2 and also
Theorem 1.4.10).

Another way to produce $t$-structures is via {\em torsion theories}.
For an application of this in noncommutative geometry see \cite{Po}.

\separate

\noindent A {\bf torsion theory} on an abelian category $\A$ is a
pair $(\bbT,\bbF)$ of full additive subcategories of $\A$ such that:

\vspace{0.1in}

\begin{itemize}
  \item  For every $T \in \bbT$ and $F\in \bbF$, we have
  $\Hom(T,F)=0$.
\vspace{0.1in}
  \item For every $A \in \A$, there is a (necessarily unique) exact
  sequence
    $$0 \to T \to A \to F \to 0, \ \ \ \ T\in \bbT,\ F\in\bbF.$$
\end{itemize}

\vspace{0.1in}
\begin{ex}
  Let $\A=\mathbf{Ab}$ be abelian categories. Take $\bbT$ to be the
  torsion abelian groups and $\bbF$ the torsion-free abelian groups.
\end{ex}

\noindent {\em Exercise.} Prove that for a torsion theory
$(\bbT,\bbF)$ we have $\bbT^{\bot}=\bbF$ and $^{\bot}\bbF=\bbT$,
that is,
   $$\bbF=\{F\in \A \ | \ \forall T\in \bbT, \ \Hom(T,F)=0\},$$
   $$\bbT=\{T\in \A \ | \ \forall F\in \bbF, \ \Hom(T,F)=0\}.$$

\separate

Out of a given torsion theory $(\bbT,\bbF)$ on an abelian category
$\A$ we can construct a new abelian category $\B$ in which the roles
of $\bbT$ and $\bbF$ are interchanged.

\vspace{0.1in}

 We construct $\B$ as follows.  Consider the following
$t$-structure on  $\mathcal{D}=\mathcal{D}(\A)$:
$$\begin{array}{rcl}
   \bbD^{\leq 0} & = & \{ B \in \mathcal{D}^{\leq 0} \ | \ H^0(B) \in
     \bbT\}, \\
   \bbD^{\geq 0} & = & \{ B \in \mathcal{D}^{\geq -1} \ | \ H^{-1}(B) \in
    \bbF\}.
\end{array}$$
\noindent This is easily seen to be a $t$-structure, so, by Theorem
\ref{T:heart}, $\B:=\bbD^{\leq 0}\cap \bbD^{\geq 0}$ is an abelian
category. More precisely,  $\B$ is equivalent to the full
subcategory of $\D(\A)$ consisting of  complexes $d \: A^{-1}\to
A^0$ such that $\ker d \in \bbF$ and $\coker d \in \bbT$.

\begin{rem} Let $\bbT'$ and $\bbF'$ be essential images of
$\bbF[1]$ and $\bbT$ in $\B$. Then $(\bbT',\bbF')$ is a torsion
theory for $\B$. (However, if we repeat the same process as above
for this torsion theory, we may not get back $\A$.)
\end{rem}

\pagebreak

\section{Appendix I: topological triangulated categories}

\noindent {\small References: (\cite{Ho}, Section 7),
\cite{Margolis}, (\cite{Weibel}, Section 10.9).}

\vspace{0.1in}

In our examples, we considered only triangulated categories that
come from algebra (e.g, derived categories). There are certain
triangulated categories that arise from topology (e.g, stable
homotopy categories) that we have not discussed in these notes but
which are worthy of attention. In line with our
topological-vs-chain-complex comparison picture of Lecture 2,
 we say a few words about the homotopy category of spectra.

\separate

\noindent {\bf Spectra.} Recall from Lecture 2 that the cone
construction (which is the main input in the construction of derived
categories) was motivated by a topological construction on spaces.
One may wonder then whether one could repeat the construction of the
derived categories in the topological setting. The problem that
arises here is that {\em the suspension functor $X \mapsto SX$ is
not an auto-equivalence} of the category of topological spaces
(while its chain complex counterpart, the shift functor $C \mapsto
C[1]$, is). To remedy this, one can formally ``invert'' the
suspension functor. This gives rise to the  category of
Spanier-Whitehead spectra (\cite{Margolis}, Chapter 1) from which we
can construct a triangulated category by inverting weak
equivalences, in the same way we obtained the derived category by
inverting quasi-isomorphisms.


 What made topologists unhappy about the category of Spanier-Whitehead spectra
 is that it is not closed under arbitrary colimits.
 \footnote{Actually neither are the {\em bounded} derived categories!}
 Nowadays there are better
 categories of spectra that do not have this deficiency. They go under
 various names
 such as {\em spectra}, {\em $\Omega$-spectra}, {\em symmetric spectra},
 etc. All these
 categories come equipped with a notion of weak equivalence, and inverting
 the weak
 equivalences in any of these categories gives rise to the {\em same}
 triangulated category,
 the {\em homotopy category of spectra}. Understanding the homotopy category
 of spectra
 is an area of active research that is called {\em stable homotopy theory}.
 To see how
 triangulated categories emerge in this context
 consult the given references.

\section{Appendix II: different illustrations of $\mathbf{TR4}$}{\label{S:TR4}}

One way of illustrating  $\mathbf{TR4}$ is via the following
``braid'' diagram:
$$\xymatrix@R=12pt@C=12pt@M=5pt{
     && X  \thickar{}{dr}{_f}  \dentar{@/^1pc/}{rr}{^{gf}}
         && Z \dentar{}{rd}{} \ar@/^1pc/[rr]  &&
            W \ar[rd]  \ar@/^1pc/@{..>}[rr]&& U[1] && \\
     & \dots && Y \ar[ru]_g   \thickar{}{rd}{}
         && V \ar@{..>} [ru] \dentar{}{rd}{}  && Y[1]  \thickar{}{ru}{}
                                   \ar[rd] && \dots & \\
     && W[-1] \ar[ru]    \ar@/_1pc/@{..>}[rr]
         && U \ar@{..>}[ru]   \thickar{@/_1pc/}{rr}{}
             &&  X[1] \thickar{}{ru}{}  \dentar{@/_1pc/}{rr}{}  && Z[1] && }$$

The axiom $\mathbf{TR4}$ can be stated as saying that, given the
three solid strands of braids, the dotted strand can be filled so
that   the diagram commutes. It is understood that each of the four
long sequences is exact (i.e., obtained from a distinguished
triangle). Note that we are requiring the existence of only three
consecutive dotted arrows; the rest of the sequence is uniquely
determined by these three.

Another depiction of the axiom $\mathbf{TR4}$ is via the following
octahedron:
$$\xymatrix@R=6pt@C=8pt@M=5pt{  & & & Z \dentar{}{rdd}{} \ar[dddlll]  & & & \\
    & & & & & & \\
    & & & & V \dentar{}{rrd}{_{+1}} \ar@{..>}[lllld] |!{[ddll];[uul]}\hole& & \\
     W \ar[rrd]^{+1}  \ar@{..>}[dddrrr]_{+1} & & & & & &
                       X \thickar{}{lllld}{^(0.4)f}  \dentar{}{llluuu}{_{gf}}\\
    & &                  Y \thickar{}{ddr}{}  \ar[uuuur]_(0.3)g  & & & &     \\
    & & & & & & \\
    & & & U \thickar{}{uuurrr}{_{+1}}  \ar@{..>}[uuuur] |!{[uuurrr];[uul]}\hole
                                                                         & & & }$$

The axiom requires that the dotted arrows could be filled. It is
understood that the four mixed-color triangles are commutative, and
the four unicolor triangles are exact.

\providecommand{\bysame}{\leavevmode\hbox
to3em{\hrulefill}\thinspace}
\providecommand{\MR}{\relax\ifhmode\unskip\space\fi MR }
\providecommand{\MRhref}[2]{%
  \href{http://www.ams.org/mathscinet-getitem?mr=#1}{#2}
} \providecommand{\href}[2]{#2}

\end{document}